\documentclass[12pt]{article}

\usepackage[left={3cm}, right= {3.5cm}, bottom={3cm}, top={3cm}]{geometry}
		
\usepackage{amsmath,amscd, amssymb}
\usepackage{mathtools}
\usepackage{amsthm} 
\usepackage{amssymb}
\usepackage{amsmath}
\usepackage[all,cmtip]{xy}

\newdir{ >}{{}*!/-5pt/@{>}}
\newdir^{ (}{{}*!/-2.5pt/@^{(}}
\newdir_{ (}{{}*!/-2.5pt/@_{(}}
\newdir{< }{{}*!/5pt/@{<}}
\newdir^{) }{{}*!/2.5pt/@^{)}}
\newdir_{) }{{}*!/2.5pt/@_{)}}

\usepackage{stmaryrd}

\usepackage{bussproofs}

\newcommand{\foot}[1]{\footnote{\linespread{1}\selectfont \,#1}}

\newcommand{\Omeag}{\Omega}

\newcommand{\gmid}{\Gamma\mid}

\newcommand{\pow}[1]{\mathcal{P}(#1)}

\newcommand{\Hom}[3]{\mathrm{Hom}_{#1}(#2,#3)}

\newcommand{\cat}[1]{\textbf{#1}}

\newcommand{\prs}[1]{\cat{Sets}^{\cat{#1}^{op}}}

\newcommand{\srp}[1]{\cat{Sets}^{\cat{#1}}}
\newcommand{\srpp}[1]{\cat{Sets}^{|\cat{#1}|}}
\newcommand{\tps}[1]{\mathcal{#1}}

\newcommand{\sub}[2]{\mathrm{Sub}_{#1}(#2)}

\newcommand{\subt}[2]{\mathrm{Sub}_{\tps{#1}}(#2)}
\newcommand{\inv}[1]{#1^{-1}}
\newcommand{\sem}[1]{\llbracket #1 \rrbracket} 
\newcommand{\y}{\textbf{y}}
\newcommand{\mtext}[1]{\ \text{#1}\ }

\newcommand{\srow}[2]{\sem{{#1}_1} #2 \dots #2 \sem{{#1}_n}}

\newcommand{\typrow}[3]{{#1}_1 : {#2}_1 #3 \dots #3 {#1}_n : {#2}_n}

\newcommand{\domp}[1]{\text{dom}^\prime}
\newcommand{\codp}[1]{\text{cod}'}

\newcommand{\compe}[1]{\text{comp}}
\newcommand{\compp}[1]{\text{comp}'}

\newcommand{\snd}[1]{\text{snd}}
\newcommand{\fst}[1]{\text{fst}}
\newcommand{\sndp}[1]{\text{snd}'}
\newcommand{\fstp}[1]{\text{fst}'}

\newcommand{\ops}[1]{\mathcal{O}(#1)}

\newcommand{\tr}{\textsf{t}}

\newcommand{\mathiff}{\ \text{ iff }\ }

\newcommand{\wrt}{w.r.t.\ } 
\newcommand{\eg}{e.g.\ }
\newcommand{\ie}{i.e.\ }

\newcommand{\secref}[1]{section \ref{#1}}
\newcommand{\lemref}[1]{lemma \ref{#1}}

\newcommand{\spf}[5]{#1 \ar@<1.5ex>[rr]^-{#3} \ar@<-1.5ex>[rr]_-{#5} && #2 \ar[ll]|-{#4}}

\newtheorem{thm}{Theorem}[section]
\newtheorem{lem}[thm]{Lemma}
\newtheorem{cor}[thm]{Corollary}
\newtheorem{prop}[thm]{Proposition}
\newtheorem{mydef}[thm]{Definition}
\newtheorem{fact}[thm]{Fact}

\theoremstyle{remark}
\newtheorem{rem}[thm]{Remark}
\newtheorem{exam}{Example}[section]
\newtheorem{exams}{Examples}[section]

\newcommand{\app}[2]{\textsf{app}(#1,#2)}

\newcommand{\soc}[1]{\Omega_{\tps{#1}}}
\newcommand{\gsoc}[2]{{#1}_\ast\soc{#2}}

\newcommand{\sh}[2]{\text{Sh}_{#1}(#2)}

\newcommand{\bigmeet}{\forall}

\begin{document}

%
%
%

\title{Topos Semantics for Higher-Order Modal Logic}
\author{Steve Awodey\thanks{
Department of Philosophy, Carnegie Mellon University}
 \and
 Kohei Kishida\thanks{
Department of Computer Science, University of Oxford}
 \and
 Hans-Christoph Kotzsch\thanks{
Munich Center for Mathematical Philosophy, LMU Munich}
}
\maketitle

\begin{quote}
\textsc{Abstract.} We define the notion of a model of higher-order modal logic in an arbitrary elementary topos $\tps{E}$.  In contrast to the well-known interpretation of (non-modal)  higher-order logic, the type of propositions is not interpreted by the subobject classifier $\soc{E}$, but rather by a suitable complete Heyting algebra $H$. The canonical map relating $H$ and $\soc{E}$ both serves to interpret equality and provides a modal operator on $H$ in the form of a comonad. Examples of such structures arise from surjective geometric morphisms $f : \tps{F} \rightarrow \tps{E}$, where $H = f_\ast\soc{F}$. The logic differs from non-modal higher-order logic in that the principles of functional and propositional extensionality are not longer valid but may be replaced by modalized versions. The usual Kripke, neighborhood, and sheaf semantics for propositional and first-order modal logic are subsumed by this notion. 
\end{quote}

\section*{Introduction}

In many conventional systems of semantics for quantified modal logic, models are built on presheaves.
Given a set $K$ of ``possible worlds", Kripke's semantics \cite{kripke63}, for instance, assigns to each world $k \in K$ a domain of quantification $P(k)$ --- regarded as the set of possible individuals that ``exist'' in $k$ --- and then $\exists x \, \varphi$ is true at $k$ iff some $a \in P(k)$ satisfies $\varphi$ at $k$.
David Lewis's counterpart theory \cite{lewis68} does the same (though it further assumes that $P(k)$ and $P(l)$ are disjoint for $k \neq l \in K$).
Such an assignment $P$ of domains to worlds is a presheaf $P : K \to \cat{Sets}$ over the set of worlds, thus an object of the topos $\cat{Sets}^K$.
(Due to the disjointness assumption one may take counterpart theory as using objects of the slice category $\cat{Sets}/K$, which however is categorically equivalent to $\cat{Sets}^K$.)
Kripke-sheaf semantics for quantified modal logic \cite{goldblatt79, ghilardi89, ghilardimeloni88, shehtmanskortsov91} is another example of this sort.
Indeed, both counterpart theory and Kripke-sheaf semantics interpret unary formulas by subsets of the ``total set of elements" $\sum_{k \in K} P(k)$, and, more generally, $n$-ary formulas by relations of type:
\[
\pow{\sum_{k \in K} P(k)^n} \cong \sub{\cat{Sets}^K}{P^n}.
\]
In fact, counterpart theory and Kripke-sheaf semantics interpret the non-modal part of the logic in the same way. (Kripke's semantics differs somewhat in interpreting $n$-ary formulas instead as subsets of $K \times (\bigcup_{k \in K} P(k))^n$.)

Among these presheaf-based semantics, the principal difference consists in how to interpret the modal operator $\Box$.
Let \cat{K} be a set of worlds $K =|\cat{K}|$ equipped with a relation $k\leq j$ of ``accessibility''.
Kripke declares that an individual $a \in \bigcup_{k \in \cat{K}} P(k)$ satisfies a property $\Box \varphi$ at world $k \in \cat{K}$ iff $a$ satisfies $\varphi$ in all $j\geq k$.
Lewis instead introduces a ``counterpart'' relation among individuals, and deems that $a \in P(k)$ satisfies $\Box \varphi$ iff all counterparts of $a$ satisfy $\varphi$.
We may take Kripke-sheaf semantics as giving a special case of Lewis's interpretation:
Assuming \cat{K} to be a preorder, the semantics takes a presheaf $P : \cat{K} \to \cat{Sets}$ on $\cat{K}^{op}$, and not just on the underlying set $|\cat{K}|$, so that a model comes with comparison maps $\alpha_{kj} : P(k) \rightarrow P(j)$ whenever $k \leq j$ in \cat{K}.
Then $\alpha_{kj} : P(k) \rightarrow P(j)$ gives a counterpart relation:
$\alpha_{kj}(a)$ is the counterpart in the world $j$ of the individual $a \in P(k)$, so that $a$ satisfies $\Box \varphi$ iff $\alpha_{kj}(a)$ satisfies $\varphi$ for all $j \geq k$. (Notable differences between Kripke-sheaf semantics and Lewis's are the following: In the former, $\cat{K}$ can be any preorder, whereas Lewis only considers the universal relation on $\cat{K}$.
Also, since $P$ is a presheaf, the former assumes that $a \in P(k)$ has one and only one counterpart in every $j\geq k$.)

In terms of interior operators, 
this gives an interpretation of $\Box$ on the poset $\sub{\cat{Sets}^{|\cat{K}|}}{uP}$ of sub-presheaves of $uP$ where  $u : \cat{Sets}^{\cat{K}} \to \cat{Sets}^{|\cat{K}|}$ is the evident forgetful functor.  Note that such sub-presheaves are just subsets of $\sum_{k \in K} P(k)$.  Explicitly, given the presheaf $P$ on $\cat{K}$ and any sub\emph{set} $\varphi\subseteq uP$ of elements of $\sum_{k \in K} P(k)$, then $\Box\varphi\subseteq \varphi \subseteq uP$ is the largest sub\emph{presheaf} contained in $\varphi$.

Observe that $u$ is the inverse image part of a fundamental example of a geometric morphism between toposes, namely, $u = i^*$ for the (surjective) geometric morphism
\[
i^*\dashv i_* : \srpp{K} \longrightarrow \srp{K}
\]
induced by the ``inclusion'' $i : |\cat{K}| \hookrightarrow \cat{K}$ of the underlying set $|\cat{K}|$ into \cat{K}.  In particular, $u$ is restriction along $i$.
This observation leads to a generalization of these various presheaf models to a general topos-theoretic semantics for first-order modal logic \cite{awodeykishida08, ghilardi07, makkaireyes95,reyeszolfaghari91}, which gives a model based on any surjective geometric morphism $f : \tps{F} \rightarrow \tps{E}$.
Indeed, for each $A$ in $\tps{E}$, the inverse image part $f^\ast : \tps{E} \rightarrow \tps{F}$ restricts to subobjects to give an injective complete distributive lattice homomorphism 
$\Delta_A : \sub{\tps{E}}{A}\to \sub{\tps{F}}{f^\ast A}$, which always has a right adjoint $\Gamma_A$. Composing these yields an endofunctor $\Delta_A \Gamma_A$ on the Heyting algebra $\sub{\tps{F}}{f^\ast A}$:
%
\begin{gather}
\label{eq:restrition.of.geometric.morphism}
\xymatrix@C+0.5pc{
\sub{\tps{F}}{f^\ast A}
\ar@{}[r]|-{\top}
\ar@<1ex>[r]^-{\Gamma_A}
\save !R(-.8) \ar@(ul,dl)_{\Delta_A \Gamma_A} \restore
&
\sub{\tps{E}}{A}
\ar@<1ex>[l]^-{\Delta_A}
}
\end{gather}

In the special case considered above,  the interior operation on the ``big algebra" $$\pow{\sum_{k \in \cat{K}} P(k)} \cong \sub{\cat{Sets}^{|\cat{K}|}}{u P}$$ determines the ``small algebra" $\sub{\cat{Sets}^{\cat{K}}}{P}$ as the Heyting algebra of upsets in $\sum_{k \in \cat{K}} P(k)$ (an upset $S \subseteq \sum_{k \in \cat{K}} P(k)$ is a subset that is closed under the counterpart relation:  $a \in S$ for $a \in P(k)$ implies $\alpha_{kj}(a) \in S$ for all $j \geq k$.) Moreover, 
$\Gamma_P$ is the operation giving ``the largest upset contained in \ldots'', and $\Delta_P$ is the inclusion of upsets into the powerset. In this case, the logic is ``classical", since the powerset is a Boolean algebra. 

In the general case, the operator $\Box$ is of course interpreted by $\Delta_A \Gamma_A$, which always satisfies the axioms for an S4 modality, since $\Delta_A \Gamma_A$ is a left exact comonad.  The specialist will note that both $\Delta_A$ and  $\Gamma_A$ are natural in $A$, in a suitable sense, so that this interpretation will satisfy the  Beck-Chevalley condition required for it to behave well with respect to substitution, interpreted as pullback (see \cite{awodeykishida08}).

This, then, is how topos-theoretic semantics generalizes Kripke-style and related semantics for quantified modal logic (cf.~\cite{awodeykishida08}).  Now let us further observe that, since $\tps{F}$ is a topos, it in fact has enough structure to also interpret \emph{higher-order} logic, and so a geometric morphism $f : \tps{F} \to \tps{E}$ will interpret \emph{higher-order modal logic}.  This is the logic that the current paper investigates.  The first step of our approach is to observe that, because higher-order logic includes a type of ``propositions", interpreted by a subobject classifier $\Omega$, the natural operations on the various subobject lattices in \eqref{eq:restrition.of.geometric.morphism} can be internalized as operations on $\Omega$.  Moreover, the relevant part of the geometric morphism $f : \tps{F} \rightarrow \tps{E}$, giving rise to the modal operator, can also be internalized, so that one really just needs the topos $\tps{E}$ and a certain algebraic structure on its subobject classifier $\Omega_\tps{E}$.  
That structure replaces the geometric morphism $f$ by the induced operations on the internal algebras $f_*\Omega_\tps{F}$ and $\Omega_\tps{E}$ inside the topos $\tps{E}$.  More generally, the idea is to describe a notion of an ``algebraic" model inside a topos $\tps{E}$, using the fact that S4 modalities always occur as adjoint pairs between suitable algebras.

In a bit more detail, the higher-order logical language will be interpreted \wrt a complete Heyting algebra $H$ in $\tps{E}$, extending ideas from traditional algebraic semantics for intuitionistic logic \cite{rasiowasikorski, scott08}.
The modal operator on $H$ arises from an (internal) adjunction 
%
\begin{gather}
\xymatrix@C+0.5pc{
H
\ar@{}[r]|-{\top}
\ar@<1ex>[r]^-{\tau}
\save !R(-.2) \ar@(ul,dl)_{i\tau} \restore
&
\Omega_\tps{E}
\ar@<1ex>[l]^-{i}
}
\end{gather}
where $i$ is a monic frame map and $\tau$ classifies the top element of $H$. (This, of course, is just the unique map of \emph{locales} from $H$ to the terminal locale.)
Externally, for each $A\in \tps{E}$, we then have a natural adjunction between Heyting algebras,
\begin{gather}
\label{diag:internalmodality}
\xymatrix@C+0.5pc{
\Hom{\tps{E}}{A}{H}
\ar@{}[r]|-{\top}
\ar@<1ex>[r]^-{\tau_A}
\save !R(-.8) \ar@(ul,dl)_{i_A\tau_A} \restore
&
\sub{\tps{E}}{A}
\ar@<1ex>[l]^-{i_A}
}
\end{gather}
defined by composition as indicated in the following diagram:
\[
\xymatrix{
& H \ar@<1ex>[d]^\tau \ar@{}[d]|\dashv \\
A \ar@/^.5pc/[ur] \ar[r] &\soc{E}\ar@<1ex>[u]^i\\
}
.
\]

Comparing \eqref{diag:internalmodality}  to \eqref{eq:restrition.of.geometric.morphism}, we note  that in the case $H=f_*\Omega_\tps{F}$ for a geometric morphism $f : \tps{F}\to\tps{E}$ we have:
\[
 \Hom{\tps{E}}{A}{H} = \Hom{\tps{E}}{A}{f_*\Omega_\tps{F}} \cong \Hom{\tps{F}}{f^*A}{\Omega_\tps{F}} \cong \sub{\tps{F}}{f^*A},
 \]
as required.

In this way, the topos semantics formulated in terms of a geometric morphism $f : \tps{F} \rightarrow \tps{E}$ gives rise to an example of the required ``algebraic" structure \eqref{diag:internalmodality}, with $H = f_\ast \Omega_\tps{F}$, and the same semantics for first-order modal logic can also be defined in terms of the latter.  On the other hand, every algebraic model in $\tps{E}$ arises in this way from a geometric morphism from a suitable topos $\tps{F}$, namely the topos of internal sheaves on $H$.
Thus, as far as the interpretation of first-order logic goes, the algebraic approach is equivalent to the  geometric one (the latter restricted to \emph{localic} morphisms, which is really all that is relevant for the interpretation). The advantage of the algebraic approach for \emph{higher}-order logic will become evident in what follows.  To give just one example,  we shall see how the interpretation results in a key new (inherently topos-theoretic) treatment of equality which illuminates the relation between modality and intensionality.

The goal of this paper is both to present the new idea of algebraic topos semantics for higher-order modal logic and to revisit the accounts of first-order semantics that are scattered in the literature, putting them into perspective from the point of view of the unifying framework developed here. The question of completeness will be addressed in a separate paper \cite{awodeykotzsch}, extending the result in \cite{awodeykishida12}.

In the remainder of this paper, we first review the well-known topos semantics for (intuitionistic) higher-order logic and describe the adjunction $i \dashv \tau$ in some detail.
The second section then states the formal system of higher-order modal logic that is considered here and gives the definition of its models.
The third section discusses in detail the failure of the standard extensionality principles and the soundness of the modalized versions thereof.
We then show how the semantics based on geometric morphisms can be captured within the present, algebraic framework.
The last section states the representation theorem mentioned above.

For general background in topos theory (particularly for \secref{ReprAlgMod}) we refer the reader to \cite{johnstone02,JoyalTierney84, maclanemoerdijk92}, and for background on higher-order type theory to \cite{jacobs99, johnstone02, lambekscott88}.
We assume some basic knowledge of category-theoretical concepts, but will recall essential definitions and proofs so as to make the paper more accessible. The algebraic approach pursued here was first investigated by Hans-J\"org Winkler and the first author, and some of these results were already contained in \cite{joergsthesis}.  Finally, we have benefitted from many conversations with Dana Scott, whose ideas and perspective have played an obvious role in the development of our approach.


\section{Frame-valued logic in a topos}

Recall that a \emph{topos} $\tps{E}$ is a cartesian closed category with equalizers and a subobject classifier $\Omega_\tps{E}$. The latter is defined as an object $\Omega_\tps{E}$ together with an isomorphism 
\begin{equation}\label{DefSOC}
\sub{\tps{E}}{A} \cong \Hom{\tps{E}}{A}{\Omega_\tps{E}},
\end{equation}
natural in $A$ (\wrt  pullback on the left, and precomposition on the right). 
Equivalently, there is a distinguished monomorphism $\top : 1 \rightarrow \soc{E}$ such that for each subobject $M \rightarrowtail A$  there is a unique map $\mu : A \rightarrow \soc{E}$ for which $M$ arises as the pullback of $\top$ along $\mu$:
\[
\xymatrix{
M \ar[r] \ar@{ >->}[d] &1 \ar[d]^\top \\
A \ar[r]_\mu & \soc{E}
}
\]
This definition determines $\Omega_\tps{E}$ up to isomorphism. 
The map $\mu$ is called the \emph{classifying map} of $M$. The category \cat{Sets} is a topos with subobject classifier the  two-element set $\cat{2}$. The classifying maps are the characteristic functions of subsets of a given set $A$. We list some further examples that will play a role later on.

\begin{exam}
An important example is the subobject classifier in the topos of $I$-indexed families of set, for some fixed set $I$; equivalently the functor category $\cat{Sets}^I$. It is a functor $\Omega : I \rightarrow \cat{Sets}$ with components $\Omega(i) = \cat{2}$.

The subobject classifier in $\prs{C}$, for any small category $\cat{C}$, is described as follows. For any object $C$ in \cat{C}, $\Omega(C)$ is the set of all \emph{sieves} $\sigma$ on $C$, \ie sets of arrows $h$ with codomain $C$ such that $h \in \sigma$ implies $h \mathbin{\circ} f \in \sigma$, for all $f$ with $\text{cod}(f) = \text{dom}(h)$. For an arrow $g : D \rightarrow C$ in \cat{C}, $\Omega(g)(\sigma)$ is the restriction of $\sigma$ along $g$:
\[
\Omega(g)(\sigma) = \{f : X \rightarrow D \mid g \mathbin{\circ} f \in \sigma\},
\]
which is a sieve on $D$. The mono $\top : 1 \rightarrow \Omega$ is the natural transformation whose components pick out the maximal sieve $\top_C$ on $C$, \ie the set of all arrows with codomain $C$ (the terminal object $1$ being pointwise the singleton).
The classifying map $\chi_m$ of a subfunctor $m : E \rightarrowtail F$ has components
\[
(\chi_m)_C(a) = \{f : X \rightarrow C \mid F(f)(a) \in E(X)\}.
\]
In particular, if \cat{C} is a \emph{preorder}, then $\Omega(C)$ is the set of all downward closed subsets of $\downarrow C$. Since in this case there is at most one arrow $g : D \rightarrow C$, the function $F(g)$ may be thought of as the \emph{restriction} of the set $F(C)$ to $F(D)$ along the inequality $D \leq C$.
\hfill $\Box$
\end{exam}

\normalsize

Each $\soc{E}$ is a complete Heyting algebra, internal in $\tps{E}$. Generally, the notion of Heyting algebra makes sense in any category with finite limits. It is an object $H$ in $\tps{E}$ with maps
\[
\xymatrix{
1 \ar[r]^{\top,\bot} & H & H \times H \ar[l]_-{\land, \lor, \Rightarrow}\\
}
\]
that provide the Heyting structure on $H$. These maps are to make certain diagrams commute, corresponding to the usual equations defining a Heyting algebra. For instance, commutativity of 
\[
\xymatrix{
H \times 1 \ar[r]^{1 \times \top} \ar[dr]_{\pi_1} & H \times H \ar[d]^\land\\
& H\\ 
}
\]
corresponds to the axiom $x \land \top = x$, for any $x \in H$. The correspondence between the usual equational definition and commutative diagrams in a category $\cat{C}$ can be made precise using the internal language of $\cat{C}$ \cite{maclanemoerdijk92}.

The induced partial ordering on $H$ is constructed as the equalizer
\[
 \xymatrix{
 E \ar@{ >->}[r] & H \times H \ar@<0.75ex>[r]^-{\land} \ar@<-0.75ex>[r]_-{\pi_1} & H ,
}	
 \]
corresponding to the usual definition 
\[
x \leq y \mathiff x \land y = x.
\]
 
The description of arbitrary joins and meets additionally requires the existence of exponentials and is an internalization of how set-indexed joins and meets in set-structures can be expressed via a suitable adjunction. For any object $I$ in $\tps{E}$, there is an arrow 
\[
\Delta_I : H \longrightarrow H^I
\]
that is the result of applying the functor $H^{(-)}$ to the unique map $I \longrightarrow 1_\tps{E}$ in $\tps{E}$. 
In detail, $\Delta_I : H \rightarrow H^I$ is the exponential transpose of $\pi_1 : H \times I \rightarrow H$ across the adjunction $(-) \times I \dashv (-)^I$. 
Set-theoretically, for any $x \in H$, $\Delta_I(x)(i) = x$, for all $i \in I$. The object $H^I$ inherits a poset structure (in fact, a Heyting structure) from $H$, which set-theoretically translates into the pointwise ordering.

$I$-indexed joins $\bigvee_I$ and meets $\bigwedge_I$ are given by internal left and right adjoints to $\Delta_I$, respectively. After all, joins and meets are coproducts and products in the Heyting algebra $H$, and these can always be defined by adjoints in exactly that way, regarding $H$ as in internal category in $\tps{E}$. This is analogous to externally defining $I$-indexed products (coproducts) of families $(A_i)_{i \in I}$ of objects in a category $\cat{C}$ by right (left) adjoints to the functor 
\[
\Delta_I : \cat{C} \longrightarrow \cat{C}^I
\]

\begin{exam}\label{Ex1}
In case $\tps{E} = \cat{Sets}$, the right adjoint $\forall_I$ to $\Delta_I$ is explicitly computed as
\begin{equation}\label{Rad1}
\forall_I(f) = \bigvee \{a \in H \mid \Delta_I(a) \leq f\},
\end{equation}
following the standard description of the right adjoint to a map of complete join-semilattices, in this case $\Delta_I$. In fact, it is not hard to see that 
\[
\forall_I(f) = \bigwedge_{i \in I} f(i).
\]
The left adjoint $\exists_I \dashv \Delta_I$ is described dually.
\hfill $\Box$
\end{exam}

\begin{exam}
An important case that will be useful later is where the category in question is of the form $\prs{C}$, for a small category $\cat{C}$. Products in $\prs{C}$ are computed pointwise. In particular, a Heyting algebra $H$ in $\prs{C}$ has pointwise natural structure. That is to say, each $H(C)$, for $C$ in \cat{C}, is a Heyting algebra in such a way that \eg for all binary operations $\star$ on $H$, $H(f) \mathbin{\circ} \star_D = \star_C \mathbin{\circ} (H(f) \times H(f))$, for any arrow $f : C \rightarrow D$ in \cat{C}. This is because the structure maps, being arrows in $\prs{C}$, are natural transformations. Naturality  in particular means that for each $f : C \rightarrow D$ in $\cat{C}$, the map $H(f)$ preserves the Heyting structure.

By contrast, exponentials are not computed pointwise but by the formulas
\begin{gather*}
H^I(C) = \Hom{}{\y C \times I}{H} \\
H^I(f) : \eta \mapsto \eta \mathbin{\circ} (\y f \times 1_I),
\end{gather*}
where $\y C$ denotes the contravariant functor $\Hom{\cat{C}}{-}{C}$. 
The induced Heyting structure on $H^I$ is the pointwise one at each component. In particular, for any $\eta , \mu : \y C \times I \rightarrow H$, 
\begin{align*}
\eta \leq \mu\ (\text{in } H^I(C))&  \mathiff \eta_D \leq \mu_D, \text{ for each }\ D \in \cat{C} \\
& \mathiff\eta_D(f,b) \leq \mu_D(f,b)\ (\text{in}\ H(D)), \text{for each } f : D \rightarrow C, b \in I(D).
\end{align*}

Since we are mainly interested in adjoints between ordered structures, for any two order-preserving maps $\eta : H \leftrightarrows G : \mu$ between internal partial orderings $H,G$ in $\prs{C}$, $\eta \dashv \mu$ means that $\eta_C \dashv \mu_C$ at each component $C$. That is to say
\[
\eta_C(x) \leq y \mathiff x \leq \mu_C(y),
\]
for all $x \in H(C)$, $y \in G(C)$.

The natural transformation $\Delta_I : H \rightarrow H^I$ (henceforth $\Delta$) determines for each $x \in H(C)$ a natural transformation $\Delta_C(x) : \y C \times I \rightarrow H$ with components
 \[
\Delta_C(x)_D(f, a) = H(f)(x).
\]
Its right adjoint $\forall_I : H^I \rightarrow H$ (henceforth $\forall$) has components, for any $\eta \in \Hom{}{\y C \times I}{H}$,
\[
\forall_C(\eta) =  \bigvee \{s \in H(C) \mid H(f)(s) \leq \eta_D(f,b),\ \text{for all}\ f : D \rightarrow C, b \in I(D)\},
\] 
where the join is taken in $H(C)$.
Dually, the left adjoint $\exists$ of $\Delta$ has components
\[
\exists_C(\eta) =  \bigwedge \{s \in H(C) \mid \eta_D(f,b) \leq H(f)(s),\ \text{for all}\ f : D \rightarrow C, b \in I(D)\}.
\]
(Note that for instance the condition on the underlying set of the join $\forall_C(\eta)$ expresses that $\Delta_C(s) \leq \eta$ as elements in $H^I(C)$, so these definitions are in accordance with the general definition of right adjoints to $\Delta_C$ given in the previous example.)

Lastly, each $H(C)$ really is a complete Heyting algebra in the usual sense of having arbitrary set-indexed meets and joins (so the previous definitions of $\forall$ and $\exists$ actually make sense). For any set $J$, the right adjoint $\forall_J : H(C)^J \longrightarrow H(C)$ can be found as follows. Consider the constant $J$-valued functor $\Delta J$ on \cat{C} (and constant value $1_J$ on arrows in \cat{C}). For any $C$ in \cat{C}, there is an isomorphism 
\[
\Hom{\cat{Sets}}{J}{HC} \cong \Hom{\widehat{\cat{C}}}{\y C \times \Delta J}{H}
\]
(natural in $J$ and $H$). Given a function $h : J \rightarrow HC$, define a natural transformation $\nu h : \y C \times \Delta J \rightarrow H$ to have components $(\nu h)_D(g,a) = H(g)f(a)$. Conversely, given a natural transformation $\eta$ on the right, define a function $f\eta :  J \rightarrow HC$ by $f\eta(a) = \eta_C(1_C,a)$. These assignments are mutually inverse. Moreover, the map that results from composing $\Delta_J : HC \rightarrow H^{\Delta J}(C)$ with that isomorphism is computed as 
\[
f(\Delta_C(x))(a) = \Delta_C(x)_C(1_C,a) = H(1_C)(x) = x,
\]
so that for any $x \in HC$, $\Delta_C(x)$ is the constant $x$-valued map on $J$. This justifies taking the right adjoint to $\Delta_J$ as the sought right adjoint of the diagonal map $HC \rightarrow H(C)^J$. 

Indeed, for exponents $\Delta J$ the formula for the right adjoint to $\Delta_C$, for instance, takes the familiar form met in the previous example
\[
\forall_C(\eta) =  \forall_J(f\eta) = \bigwedge_{a \in J} f\eta(a) = \bigwedge_{a \in \Delta J(C)}  \eta_C(1_C,a),
\]
or 
\[
\forall_J(h) = \forall_C(\nu h) = \bigwedge_{a \in \Delta J(C)}  (\nu h)_C(1_C,a) = \bigwedge_{a \in J} h(a),
\] 
respectively.\hfill $\Box$
\\
\end{exam}


\normalsize

For convenience, let us recall the Heyting structure of $\soc{E}$ in more detail, as it will be useful later on.  It is uniquely determined by the natural isomorphism \eqref{DefSOC} and the Yoneda lemma which ``internalizes'' the (complete) Heyting structure of $\Hom{\tps{E}}{-}{\soc{E}}$ (coming from $\sub{\tps{E}}{-}$) to $\soc{E}$. Since each pullback functor $f^\ast : \sub{}{B} \rightarrow \sub{}{A}$, for $f : A \rightarrow B$ in $\tps{E}$, preserves the Heyting structure on $\sub{}{B}$, all the required diagrams that define the Heyting operations on $\soc{E}$ necessarily commute.

The top element is $\top : 1 \rightarrow \soc{E}$, which by the previous considerations is the classifying map of the identity on the terminal object. The bottom element is the characteristic map of the monomorphism $0 \rightarrowtail 1$, where 0 is the initial object of $\tps{E}$. Meets 
\[
\land : \soc{E} \times \soc{E} \longrightarrow \soc{E}
\]
are given as the classifying map of  $\langle \top , \top \rangle : 1 \longrightarrow \soc{E} \times \soc{E}$, which is the classifying map of the pullback of $\langle 1 , \top u  \rangle$ and $\langle \top u  , 1 \rangle$ ($u : \soc{E} \rightarrow 1$ is the canonical map):
\[
\xymatrix{
1 \ar[rr]^\top \ar[d]_\top && \soc{E} \ar[d]|{\langle 1 , \top u  \rangle}\\
\soc{E} \ar[rr]_-{\langle \top u  , 1 \rangle} && \soc{E} \times \soc{E} 
}
\]
viewed as subobject of $\soc{E} \times \soc{E}$; while $\langle 1 , \top u  \rangle$ and $\langle \top u  , 1 \rangle$ in turn arise as the subobjects classified by $\pi_2$ and $\pi_1$, respectively.
In a similar way, joins are constructed as  classifying map of the image of the map 
\[
[\langle 1 , \top u_{\soc{E}} \rangle , \langle \top u_{\soc{E}} , 1 \rangle] : \soc{E} + \soc{E} \longrightarrow \soc{E} \times \soc{E}.
\] 
 Implication is given as the classifying map of the equalizer
\[
 \xymatrix{
 E \ar@{ >->}[r] & \soc{E} \times \soc{E} \ar@<0.75ex>[r]^-{\land} \ar@<-0.75ex>[r]_-{\pi_1} & \soc{E} .  \\
}	
 \]
The classifying map can be factored as follows, where the two squares are pullbacks:
\[
\xymatrix{
E \ar[r] \ar@{ >.>}[d] & \soc{E} \ar[r] \ar[d]|{\Delta_{\soc{E}}} & 1 \ar[d]^\top\\
\soc{E} \times \soc{E} \ar[r]_{\langle \land , \pi_1 \rangle} & \soc{E} \times \soc{E} \ar[r]_-{\delta_{\soc{E}}} & \soc{E}
}
\]
following a standard description of equalizers.\foot{Actually, the Yoneda argument determines $\Rightarrow$ as the classifying map of the subobject $\forall_{\langle \top u_{\soc{E}} , 1 \rangle}(\top)$ of $\soc{E} \times \soc{E}$, where the latter is precisely the said equalizer.}

Using the Yoneda principle one also obtains indexed meets and joins as adjoints to the map $\Delta_I : \soc{E} \rightarrow \soc{E}^I$. They are essentially provided by the fact that, for any topos $\tps{E}$, and any arrow $f :A \rightarrow B$ in $\tps{E}$, the pullback functor 
 \[
 f^\ast : \sub{\tps{E}}{B} \longrightarrow \sub{\tps{E}}{A}
 \]
 has both a right and a left adjoint. Adding a parameter $X$ yields that 
 \[
 (1_X \times f)^\ast : \sub{\tps{E}}{X \times B} \longrightarrow \sub{\tps{E}}{X \times A} 
 \]
  restricts to a functor
 \[
 \Hom{\tps{E}}{X}{\soc{E}^B} \longrightarrow  \Hom{\tps{E}}{X}{\soc{E}^A}
 \]
 by the isomorphisms
 \[
 \sub{\tps{E}}{X \times Y} \cong \Hom{\tps{E}}{X \times Y}{\soc{E}} \cong \Hom{\tps{E}}{X}{\soc{E}^Y}.
 \]
 These are natural in $X$ and so by Yoneda provide a map 
 \[
 \soc{E}^B \longrightarrow  \soc{E}^A,
 \]
 which is precisely $\soc{E}^f$.

In particular, $\Delta_I$ arises in this way from pullback along the projection $\pi_1 : X \times I \rightarrow X$:
\[
\pi_1^\ast : \sub{\tps{E}}{X} \longrightarrow \sub{\tps{E}}{X \times I}, 
\]
that is by applying the previous argument to the map $u_I : I \longrightarrow 1$, as required.
The external adjoints of $\pi_1^\ast$ induce the required internal adjoints of $\soc{E}^{u_I} = \Delta_I$.
\\

\normalsize

As is well-known, one can interpret (intuitionistic) higher-order logic \wrt  this algebraic structure on $\soc{E}$ \cite{lambekscott88, maclanemoerdijk92}.  In particular, each formula $\gmid \varphi$, where $\Gamma = (\typrow{x}{A}{,})$ is a suitable variable context for $\varphi$, is recursively assigned an arrow 
\[
\srow{A}{\times} \xrightarrow{\sem{\varphi}} \soc{E}
\]
in $\tps{E}$. Connectives and quantifiers are interpreted by composing with the evident Heyting structure maps of $\soc{E}$ described above. 
\normalsize
For instance, $\sem{x : A \mid \forall y . \varphi}$ is the arrow
 \[
\sem{A} \xrightarrow{  \lambda_{\sem{B}} \sem{\varphi} } \soc{E}^{\sem{B}} \xrightarrow{\forall_{\sem{B}}} \soc{E},
 \]
 where $\lambda_{\sem{B}} \sem{\varphi}$ is the exponential transpose of 
 \[
 \sem{\varphi} : \sem{A} \times \sem{B} \longrightarrow \soc{E}.
 \]
In particular, the equality predicate on each type $M$ is interpreted as the classifying map $\delta_{\sem{M}}$ of the diagonal
 \[
 \langle 1_{\sem{M}} , 1_{\sem{M}} \rangle : {\sem{M}}  \longrightarrow {\sem{M}} \times {\sem{M}}.
 \] 


\begin{exam}
When $\tps{E} = \cat{Sets}$, and $\soc{\cat{Sets}} = \cat{2}$, then the right adjoint $\forall_I$ to $\Delta_I : \cat{2} \rightarrow \cat{2}^I$ is by definition required to satisfy 
\[
\Delta_I(x) \leq f \mathiff x \leq \forall_I(f),
\]
which holds just in case $\forall_I$ satisfies
\[
\forall_I(f) = 1 \mathiff f(i) = 1,\ \text{ for all }\ i \in I. 
\]
Equivalently,
\[
\forall_I(S) = 1 \mathiff S = I,
\]
where $S \subseteq I$. Given a formula $x : X \mid \varphi$, and an interpretation $\sem{X} \xrightarrow{\sem{\varphi}} \cat{2}$, then $\lambda_{\sem{X}}\sem{\varphi} : 1 \rightarrow \cat{2}^{\sem{X}}$ picks out the subset $S$ of $\sem{X}$ whose characteristic map is $\sem{\varphi}$, \ie the set of objects in $\sem{X}$ that satisfy $\varphi$. Thus $\sem{\forall x . \varphi} = 1$ if and only if $S = \sem{X}$, as expected.
\hfill $\Box$
\end{exam}

\normalsize

In principle, these definitions make sense for any Heyting algebra $H$ in $\tps{E}$ in place of $\soc{E}$, except for interpreting equality, since there is no notion of classifying map available for arbitrary $H$. We present below a way in general to canonically interpret equality for arbitrary $H$, closely connected to the treatment of modal operators.

\begin{mydef}
In any topos $\tps{E}$, a \emph{frame} $H$ in $\tps{E}$ is a complete Heyting algebra $H$ in $\tps{E}$. A \emph{frame homomorphism} $f : H \rightarrow G$ is a map $f$ in $\tps{E}$ that is internally $\bigvee,\land$-preserving. 
\end{mydef}

For instance in any topos $\tps{E}$ the object $1+1$ is an internal Boolean algebra, and thus a frame. Here, $\tau : 1+1 \rightarrow \soc{E}$ is the classifying map of the first coprojection.



\begin{exam}\label{FrameEx}
The prototypical frame is the collection of open sets $\ops{X}$ of a topological space $X$. The set $\ops{X}$ is a complete Heyting algebra, as is $\pow{X}$. However, arbitrary meets in $\ops{X}$ are in general not mere intersections.  That is to say, the inclusion 
\[
i : \ops{X} \hookrightarrow \pow{X}
\] 
does not preserve them. This exhibits $\ops{X}$ as a subframe of $\mathcal{P}(X)$ rather than a sub-Heyting algebra. The example also illustrates why the notion frame homomorphism matters at all. Note also that every frame map $f : H \rightarrow G$ has a right adjoint $f_\ast$, defined for any $y \in G$ as
\begin{equation}\label{TAdFM}
f_\ast(y) = \bigvee \{x \in H \mid f(x) \leq y\}.
\end{equation}
The right adjoint to the inclusion $i$ is the interior operation on the topological space $X$, which determines, in accordance with the formula for $f_\ast$, the largest open subset (\wrt $X$) of an arbitrary subset of $X$.

A related and more elementary example is the set inclusion $\cat{3} \hookrightarrow \cat{4}$ of the three element Heyting algebra into the four element Boolean algebra, as indicated in:
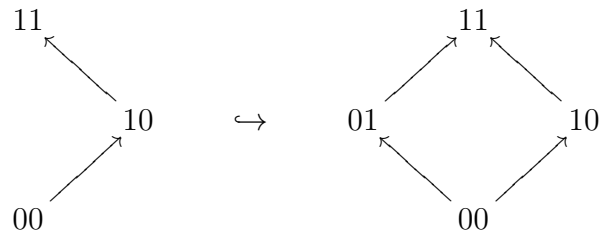
\[
\xymatrix{
&11&  &&& 11\\
&&10 \ar[ul]&\hookrightarrow&01\ar[ur]&&10\ar[ul] \\
&00\ar[ur]&&&&00\ar[ul]\ar[ur]
}
\]
(\cat{3} may be thought of as the open set structure of the Sierpi\'nski space.)
The inclusion does not preserve the implication $10 \rightarrow 00$:
\[
10 \rightarrow 00 = 00,\ \text{in}\ \cat{3}
\]
while 
\[
10 \rightarrow 00 = 01,\ \text{in}\ \cat{4}.
\]
Since $i$ preserving arbitrary meets is equivalent to saying that $i$ preserves implications,  $\cat{3}$ is included in $\cat{4}$ as a subframe rather than as a sub-Heyting algebra. \hfill $\Box$

\end{exam}

%

In a topos $\tps{E}$ the frame $\soc{E}$ plays a distinguished role:

\begin{lem}\label{InFrame}
In any topos $\tps{E}$, the subobject classifier $\soc{E}$ is the initial frame. That is to say, for every frame $H$ in $\tps{E}$, there is a unique frame map $i : \soc{E} \longrightarrow H$. Moreover, the right adjoint $\tau$ of $i$ is the classifying map of the top element $\top_H : 1 \longrightarrow H$ of $H$.
\end{lem}

We refer to \cite{johnstone02} (C1.3) for the proof.

We will  mainly be interested in those frames $H$ for which the map $i : \soc{E} \rightarrow H$ is monic, to which we will refer as \emph{faithful}.\foot{A frame $H$ is faithful in this sense iff the inverse image part of the canonical geometric morphism $\sh{\tps{E}}{H} \longrightarrow \tps{E}$ is faithful (see section \ref{ReprAlgMod}).}  This map $i : \soc{E} \rightarrow H$ will play a crucial role both in modelling equality and the modal operator on $H$.  Looking ahead, 
suppose given a suitable (intuitionistic) higher-order modal theory (as in section \ref{HOMLModels}). Then we shall interpret equality on a type $A$ \wrt an $H$-valued model in a topos $\tps{E}$ as the composite of $i$ with the usual classifying map of equality:
\[
\sem{A} \times \sem{A} \xrightarrow{\delta_{\sem{A}}} \soc{E} \xrightarrow{i} H.
\]
The semantics thus obtained is not sound \wrt standard higher-order intuitionistic logic; in particular, function and propositional extensionality fail (as we shall show by providing counterexamples).  On the other hand, one can restore soundness by taking into account the following naturally arising modal operator.

\begin{lem}
Given a frame $H$ in a topos $\tps{E}$, let $i\dashv\tau$ be the canonical adjunction described in lemma \ref{InFrame},
\[
i : \soc{E} \leftrightarrows H : \tau.
\]
The composite $i\mathbin{\circ}\tau$ is then an S4 modality on $H$.
\end{lem}

\begin{proof}
The composite $i\mathbin{\circ}\tau$ preserves finite meets because both components do.  
In virtue of $i \dashv \tau$, the composite is a comonad, which gives the S4 laws. 
\end{proof}


\section{Higher-order intuitionistic S4}\label{HOMLModels}

The formal system of higher-order modal logic considered here is simply the union of the usual axioms for higher-order logic and S4. The higher-order part is a version of type theory (cf. \cite{jacobs99, johnstone02, lambekscott88}).  Types and terms are defined recursively. A higher-order language $\mathcal{L}$ consists of a  collection of basic types $A,B, \dots$ along with basic terms (constants) $a : A, b : B$.  To stay close to topos-theoretic formulations, we assume the following type and term forming operations that inductively specify the collection of types and terms of the language:
\begin{itemize}
\item There are basic types $1$, $\textsf{P}$
\item If $A$, $B$ are types, then there is a type $A \times B$ 
\item If $A$, $B$ are types, then there is a type $A^B$
\end{itemize}

Terms are recursively constructed as follows. Here we assume, for every type $A$, an infinite set of variables of type $A$, written as $x:A$, to be given. We follow \cite{jacobs99} in writing $\gmid t : B$, for $\Gamma = (\typrow{x}{A}{,})$, involving at least all the free variables in the term $t$. A context $\Gamma$ may also be empty. Formally, every term $t$ always occurs in some variable context $\Gamma$ and is well-typed only \wrt  such a context. This is important to understand the recursive clauses below. To simplify notation, however, we omit $\Gamma$ if it is unspecified and the same throughout a recursive clause.

\begin{itemize}
\item There are distinguished terms $\emptyset \mid \ast : 1$ and $\emptyset \mid \top , \bot: \textsf{P}$
\item If $ t : A$ and $s : B$ are terms, then $ \langle t,s \rangle : A \times B$ is a term
\item If $ t : A \times B$ is a term, then there are terms $ \pi_1 t : A$ and $ \pi_2 t : B$
\item If $\gmid t : A$ is a term and $y : B$ a variable in $\Gamma$, then there is a term $\Gamma[y:B] \mid \lambda y . t : A^B$; where $\Gamma[y:B]$ is the context that results from $\Gamma$ by deleting $y:B$.
\item If $ t : A^B$ and $ s : B$ are terms, then $ \app{t}{s} : A$ is a term.
\item For any two terms $ t : \textsf{P}$, $ s : \textsf{P}$ there are terms $ t \land s : \textsf{P}$, $ t \lor s : \textsf{P}$, $ t \Rightarrow s : \textsf{P}$.
\item If $\Gamma , y : B \mid t : \textsf{P}$ is a term, then $\Gamma \mid \forall y  . t : \textsf{P}$ is a term; and similarly for $\Gamma \mid \exists y  . t : \textsf{P}$
\item If $ t : A$ and $ s : A$ are terms, then $ s =_A t : \textsf{P}$ is a term.
\item If $ t : \textsf{P}$ is a term, then $ \Box t : \textsf{P}$ is a term.
\end{itemize}

One also assumes the usual structural rules of weakening of the variable context (adding dummy variables), contraction, and permutation. We may also assume that each variable declaration occurs only once in a context.

As usual, we define a deductive system by specifying a relation $\vdash$ between terms of type $\textsf{P}$. The crucial difference between the standard formulation of intuitionistic higher-order logic and the present one are the modified extensionality principles marked with ($\ast$).

\begin{itemize}
\item $ \varphi \vdash \varphi$
\item $\displaystyle\frac{ \varphi \vdash \psi\ \ \ \  t : A}{ \varphi[t/x] \vdash \psi[t/x]}$, for $x : A$ (similarly for simultaneous substitution)

\item $\displaystyle\frac{\varphi \vdash \psi \hspace{1em}  \psi \vdash \vartheta}{ \varphi \vdash \vartheta}$
\vspace{8pt}

\item $ \top \vdash x =_A x$, where $x : A$
\item $ \varphi \land x =_A x' \vdash \varphi[x'/x]$. where $x:A, x':A$
\item[($\ast$)] $ \Box \forall x  (f(x) =_B g(x)) \vdash f =_{B^A} g$, for terms $x:A$ and $f,g : B^A$
\item[($\ast$)] $ \Box (p \Leftrightarrow q) \vdash p =_{\textsf{P}} q$, for terms $p,q: \textsf{P}$
\vspace{8pt}

\item $ \top \vdash \ast =_1 x$, where $x:1$
\item $ \top \vdash \pi_1\langle x,y \rangle =_A x$ and $ \top \vdash \pi_2\langle x,y \rangle =_B y$, where $ x :A$ and $ y : B$
\item $ \top \vdash \langle \pi_1 w , \pi_2 w \rangle =_{A \times B} w$, for $ w : A \times B$

\item $\Gamma [x:A] \mid  \top \vdash \app{\lambda x . t}{x'} =_B t[x'/x]$, for  $\gmid t : B$ and $x':A$
\item $ \top \vdash \lambda x .\app{w}{x} =_{B^A} w$, for $w : B^A$
\vspace{8pt}

\item $ \varphi \vdash \top$, for any $ \varphi : \textsf{P}$
\item $ \bot \vdash \varphi$, for any $ \varphi : \textsf{P}$
\item $ \varphi \vdash \psi \land \vartheta \mathiff  \varphi \vdash \psi$ and $\varphi \vdash \vartheta$
\item $ \varphi \lor \psi \vdash \vartheta \mathiff  \varphi \vdash \vartheta$ and $\psi \vdash \vartheta$
\item $ \varphi \vdash \psi \Rightarrow \vartheta \mathiff  \varphi \land \psi \vdash \vartheta$
\item $ \gmid \exists x.  \varphi \vdash \psi \mathiff \Gamma, x:A \mid \varphi \vdash \psi$
\item $ \gmid \varphi \vdash \forall x.  \psi \mathiff \Gamma , x:A \mid \varphi \vdash \psi$
\end{itemize}

\begin{mydef}
A \emph{theory} in a language $\mathcal{L}$ as specified above consists of a set of closed sentences $\alpha$, \ie terms of type $\textsf{P}$ with no free variables (well-typed in the empty context), and which may be used as axioms in the form $\gmid \top \vdash \alpha$. 
\end{mydef}

\begin{rem}
Adding the axiom 
\[
\gmid \top \vdash \forall p . p \lor \neg p
\] 
makes the logic classical. 

As is well-known there are more concise formulations of higher-order systems. The particular one chosen here is very close to the definition of a topos as a cartesian closed category with subobject classifier. One does not really need all exponential types and their constructors, however, but only those of the form $\textsf{P}^A$, for every type $A$, which we write $\textsf{P}A$ and call \emph{powertypes}. Along these lines one may define:
\[
\{x : A \mid \varphi\} :\equiv \lambda x. \varphi : \textsf{P}A,
\]
where $x : A \mid \varphi : \textsf{P}$. On the other hand, for $\sigma : \textsf{P}A$ and $x : A$, set 
\[
x \in \sigma :\equiv \app{\sigma}{x}.
\]
According to the axioms for exponential terms, we have
\begin{align*}
x' : A & \mid \top \vdash x' \in \{x : A \mid \varphi\} = \varphi[x'/x] \\
& \mid \top \vdash \{x : A \mid x \in w\} = w.
\end{align*}
Thus one could instead take only types of the form $\textsf{P}A$, and the constructors $\{ \dots  \mid - \}$ and $\in$ as basic, along with the last two axioms. For further simplifications see \cite{johnstone02,lambekscott88}.
\end{rem}


Finally, the S4 axioms are the usual ones
\begin{itemize}
\item $\displaystyle \frac{\gmid \varphi \vdash \psi}{\gmid \Box \varphi \vdash \Box \psi}$
\item $\gmid \top \vdash \Box \top$
\item $\gmid \Box \varphi \land \Box \psi \vdash \Box (\varphi \land \psi)$ 
\item $\gmid \Box \varphi \vdash \varphi$
\item $\displaystyle \gmid \Box \varphi \vdash \Box\Box\varphi$
\end{itemize}

The first three axioms express that $\Box$, viewed as an operator, is a monotone finite meet preserving operation.  
The other two axioms are the $T$ and 4 axioms, respectively. Further useful rules provable from the axioms are necessitation
\[
\frac{\gmid \top \vdash \varphi}{\gmid \top \vdash \Box \varphi},
\]
and the axiom $K$:
\[
\gmid \Box (\varphi \Rightarrow \psi) \vdash \Box \varphi \Rightarrow \Box \psi.
\]
\\

\noindent
Although it is essentially obvious, for the sake of completeness we provide a definition of a model of this language in a topos.

\begin{mydef}\label{DefMod}
A \emph{model} of a higher-order modal type theory in a topos $\tps{E}$ consists of a faithful frame $H$ in $\tps{E}$, and an assignment $\sem{-}$ that assigns to each basic type $A$ in $\mathcal{L}$ an object $\sem{A}$ in such a way that 
\begin{itemize}
\item $\sem{1} = 1_\tps{E}$
\item $\sem{\textsf{P}} = H$

\item $\sem{A \times B} = \sem{A} \times \sem{B}$
\item $\sem{A^B} = \sem{A}^{\sem{B}}$.
\end{itemize}
Moreover, each term $\gmid t : B$ in $\mathcal{L}$, where $\Gamma = (\typrow{x}{A}{,})$ is a suitable variable context for $t$, is assigned an arrow 
\[
\sem{t} : \sem{\Gamma} \rightarrow \sem{B}
\]
recursively as follows (where $\sem{\Gamma}$ is short for $\srow{A}{\times}$ and $\sem{t}$ really means $\sem{\gmid t : B}$).

\begin{itemize}

\item Each constant $c : A$ in $\mathcal{L}$ is assigned an arrow 
\[
\sem{c} : 1_\tps{E} \rightarrow\sem{A}.
\]
In particular:
\begin{itemize}
\item[] $\sem{\top} = \top_H : 1_\tps{E} \longrightarrow H$
\item[] $\sem{\bot} = \bot_H : 1_\tps{E} \longrightarrow H$
\item[] $\sem{\ast : 1} = 1_{1_\tps{E}}$ (the identity arrow on the terminal object).
\end{itemize}

This extends to arbitrary terms-in-context as follows

\item For any constant $c : A$,  $\sem{\gmid c : A}$ is the arrow
\[
\sem{\Gamma} \xrightarrow{u} 1_\tps{E} \xrightarrow{\sem{c}} \sem{A}
\]

\item If $\gmid s : A$ and $\gmid t : B$ are terms, then $\sem{\gmid \langle s,t \rangle : A \times B}$ is the map
\[
\langle \sem{s} , \sem{t} \rangle : \sem{\Gamma} \rightarrow \sem{A} \times \sem{B}.
\]

\item If $\gmid t : A \times B$ is a term, then $\sem{\gmid \pi_1 t : A}$ is 
\[
\sem{\Gamma} \xrightarrow{\sem{t}} \sem{A} \times \sem{B} \xrightarrow{\pi_1} \sem{A},
\]
and similarly for $\pi_2 t $.

\item If $\gmid t : A$ is a term and $y : B$ a variable in $\Gamma$, then $\sem{\Gamma[y:B] \mid \lambda y . t : A^B}$ is 
\[
{\lambda_{\sem{B}}\sem{t}} : \sem{\Gamma[y:B]} \rightarrow A^{\sem{B}}
\]

\item If $\gmid t : A^B$ and $\gmid s : B$ are terms, then $\sem{\gmid \app{t}{s} : A}$ is 
\[
\langle \sem{t},\sem{s} \rangle : \sem{\Gamma} \rightarrow A^B \times B \xrightarrow{\varepsilon} A.
\]

\item For any two terms $\gmid p : \textsf{P}$, $\gmid q : \textsf{P}$, and $\star$ any of the connectives $\land, \lor , \Rightarrow$, $\sem{\gmid p \star q : \textsf{P}}$ is
\[
\sem{\Gamma} \xrightarrow{\langle \sem{p} , \sem{q} \rangle} H \times H \xrightarrow{\star} H,
\]
where in the last line $\star$ is the evident algebraic operation on $H$.

\item If $\Gamma , y : B \mid t : \textsf{P}$ is a term, then $\sem{\Gamma \mid \forall y . t : \textsf{P}}$ is 
\[
\sem{\Gamma} \xrightarrow{\lambda_{\sem{B}} \sem{t}} H^{\sem{B}} \xrightarrow{\forall_{\sem{B}}} H
\]
and similarly for $\sem{\Gamma \mid \exists y . t : \textsf{P}}$ via $\exists_{\sem{B}}$.

\item If $\gmid t : A$ and $\gmid s : A$ are terms, then $\sem{\gmid t =_A s : \textsf{P}}$ is the map
\[
\sem{\Gamma} \xrightarrow{\langle \sem{t} , \sem{s} \rangle} \sem{A} \times \sem{A} \xrightarrow{\delta_{\sem{A}}} \soc{E} \xrightarrow{i} H,
\]
where $i$ is the unique (monic) frame map.

\item If $\gmid t : \textsf{P}$ is a term, then $\sem{\gmid \Box t : \textsf{P}}$ is the map
\[
\sem{\Gamma} \xrightarrow{\sem{t}} H \xrightarrow{\tau} \soc{E} \xrightarrow{i} H,
\]
where $\tau$ is the classifying map of $\top_H : 1 \rightarrow H$, as described before.

\end{itemize}

\end{mydef}

\noindent
Before moving on, let us review some common examples

\begin{exams}

\begin{enumerate}\strut

\item A well-studied class of examples are structures induced by surjective geometric morphisms $f : \tps{F} \rightarrow \tps{E}$. If $\tps{F}$ is Boolean, then so is $f_\ast\soc{F}$. For instance, there are geometric morphisms
\[
\cat{Sets}^{|\cat{C}|} \longrightarrow \srp{C}
\]
induced by the inclusion $|\cat{C}| \rightarrow \cat{C}$. When $\cat{C}$ is a preorder, then this yields Kripke semantics for first-order modal logic. This case was originally studied in \cite{ghilardimeloni88, shehtmanskortsov91}.

Similarly, the canonical geometric morphism  
\[
\cat{Sets}/X \longrightarrow \sh{}{X}
\]
induced by the continuous inclusion $|X| \hookrightarrow X$ gives rise to sheaf models for classical first- (and higher-) order modal logic, studied in  \cite{awodeykishida08}. The exact structure of these examples will be discussed in more detail in section \ref{RecGeomCase} below.

\item More generally, by a well-known theorem of Barr, every Grothendieck topos $\tps{G}$ can be covered by a Boolean topos $\tps{B}$ in the sense that there is a surjective geometric morphism
\[
f : \tps{B} \longrightarrow \tps{G}.
\]
For $H = \gsoc{f}{B}$, this provides models in Grothendieck topoi.\foot{Cf. \eg \cite{maclanemoerdijk92}, IX.9. Actually, the geometric morphism $f$ can be extended to a surjective geometric morphism $\tps{E} \longrightarrow \tps{B} \longrightarrow \tps{G}$, where $\tps{E}$ is the topos of sheaves on a topological space, although $\tps{E}$ might not be Boolean (\cite{maclanemoerdijk92}, IX.11).}

\item Of course, in any topos $\tps{E}$ the subobject classifier $\soc{E}$ itself would do. However, as noted \eg in \cite{reyeszawadowski93, reyeszolfaghari96}, the resulting modal operator will be the identity on~$\soc{E}$.

\end{enumerate}

\end{exams}


\section{Soundness of algebraic semantics}

The given system of intuitionistic higher-order S4 modal logic is sound \wrt the semantics described in def.\ \ref{DefMod}. Except for the two extensionality principles, soundness is straightforward following known topos semantics.
The reason why plain propositional extensionality fails in our semantics is the interpretation of implication. In the general topos semantics based on $\soc{E}$ Heyting implication on $\soc{E}$ is given by the map 
\[
\soc{E} \times \soc{E} \xrightarrow{\langle \pi_1 , \land \rangle} \soc{E} \times \soc{E} \xrightarrow{\delta} \soc{E}
\]
that immediately implies propositional extensionality. By contrast, for an arbitrary frame $H$ we observe:

\begin{lem}\label{CExPropExt}
For an arbitrary topos $\tps{E}$, and a (faithful) frame $H$ in $\tps{E}$, it is not in general the case that 
\[
\xymatrix{
H \times H \ar[r]^-\Rightarrow \ar[d]|{\langle \pi_1 , \land \rangle} & H \\
H \times H \ar[r]^-{\delta_H} & \soc{E} \ar[u]_i\\
}
\]
commutes.
\end{lem}

\begin{proof}

A counterexample may easily be found in the topos \cat{Sets} with subobject classifier \cat{2} and $H = \pow{X}$, for some set $X \neq 1$. The adjunction 
\[
i : \cat{2} \leftrightarrows \pow{X} : \tau
\]
($i \dashv \tau$)
is defined by
\[
i(x) = \begin{cases}
X, & \text{if}\  x = 1\\
\emptyset, &\text{if}\  x = 0\\
\end{cases}
\]
and
\[
\tau(U) = 1 \mathiff U = X.
\]
For any $U,V \in \pow{X}$,
\[
U \Rightarrow V = \bigcup \{W \in \pow{X} \mid W \cap U \subseteq V\}.
\] 
If $U \nsubseteq V$, then $U \neq U \cap V$, and so 
\[
i \delta\langle \pi_1,\land \rangle (U,V) = i\delta_{\pow{X}}(U, U \cap V) = i(0) = \emptyset.
\]
But $U \nsubseteq V$ does not in general imply $U \Rightarrow V = \emptyset$. (Consider \eg $V \subseteq U \Rightarrow V$, for $U \cap V \neq \emptyset$.) 
\end{proof}

As suggested by the example, the reason for the failure of plain propositional extensionality is that failure to be true (in the sense of $\top = X \nsubseteq U \Rightarrow V$) does not imply equality to $\bot$ in $H$. On the other hand, note that $\tau(U \Rightarrow V) = 0$, because $X \nsubseteq U \Rightarrow V$. This observation generalizes. Although $i\delta \langle \pi_1, \land \rangle =\ \Rightarrow$ fails in general, we have the following.

\begin{lem}\label{ImpLem}
In any topos $\tps{E}$, the diagram
\[
\xymatrix{
H \times H \ar[r]^-\Rightarrow \ar[d]|{\langle \pi_1 , \land \rangle} & H \ar[d]^\tau\\
H \times H \ar[r]^-{\delta_H} & \soc{E}
}
\]
 commutes, and thus
\[
{i \tau} \mathbin{\circ} \Rightarrow\ =\ i \delta_H \langle \pi_1 , \land \rangle.
\]
\end{lem}

\begin{proof}
Consider the pullbacks
\begin{gather*}
\xymatrix{
(\leq) \ar[r] \ar[d] & 1 \ar@{=}[r] \ar[d]^\top & 1 \ar[d]^\top \\
H \times H \ar[r]_-\Rightarrow & H \ar[r]_\tau & \soc{E}
}
\\
\xymatrix{
(\leq)  \ar[r] \ar[d] & H \ar[r] \ar[d]^\Delta& 1\ar[d]^\top \\
H \times H \ar[r]_-{\langle \pi_1 , \land \rangle} & H \times H \ar[r]_-{\delta_H} & \soc{E} 
}
\end{gather*}
whence the claim follows from uniqueness of classifying maps. The left-hand square in the first diagram is a pullback by the definition of $\Rightarrow$, while the second diagram is the definition of the induced partial ordering on $H$ as the equalizer of $\pi_1$ and $\land$.
\end{proof}

\noindent
This argument neatly exhibits the conceptual role played by the modal operator $\tau$ (more exactly, the adjunction $i \dashv \tau$). The soundness proof is essentially a corollary to that.

\begin{cor}\label{ModExtProp}
Modalized propositional extensionality
\[
p : \textsf{P}, q : \textsf{P}\ |\ \Box(p \Leftrightarrow q) \vdash p =_\textsf{P} q
\]
is true in any model $(\tps{E},H)$.
\end{cor}

\begin{proof}

In view of \lemref{ImpLem}, and since $\tau, i$ commute with meets, the left-hand side of the above sequent is interpreted as the map
\[
i \land(\delta_H \times \delta_H) \langle \langle \land_H,\pi_1 \rangle , \langle \land_H , \pi_2 \rangle \rangle,
\]
with $\land$ the meet on $\soc{E}$. The right-hand side is the internal equality on $H$:
\[
i \delta_H : H \times H \rightarrow \soc{E} \rightarrow H.
\]

It is clear from the properties of $\leq_\Omega$ as a partial ordering that
\[
\land (\delta_H \times \delta_H) \langle \langle \land_H,\pi_1 \rangle , \langle \land_H , \pi_2 \rangle \rangle \leq_{\Omega} \delta_H.
\]
Since $i$ preserves that ordering, we have
\[
i\land (\delta_H \times \delta_H) \langle \langle \land_H,\pi_1 \rangle , \langle \land_H , \pi_2 \rangle \rangle \leq_H i\delta_H.
\qedhere
\]
\end{proof}


The failure of plain function extensionality and its recovering via $\tau$ can be analyzed in a similar fashion. For non-modal function extensionality in the standard $\soc{E}$-valued setting essentially holds because $\forall_Y \mathbin{\circ} (\delta_X)^Y = \delta_{X^Y}$.
However, in our setting we don't in general have $\forall_Y \mathbin{\circ} (i\delta_X)^Y = i\delta_{X^Y}$, but rather:

\begin{lem}\label{FunExtLem}
For any topos $\tps{E}$, and any faithful frame in $H$, the following diagram commutes:
\[
\xymatrix{
\Omega^Y \ar[r]^-{i^Y} & H^Y \ar[d]^{\tau^Y} \ar[r]^-{\forall_Y} & H \ar[d]^\tau\\
X^Y \times X^Y \ar[u]^{(\delta_X)^Y} \ar[r]^-{(\delta_X)^Y} \ar@/_2pc/[rr]_{\delta_{X^Y}}& \Omega^Y \ar[r]^-{\forall_Y} & \Omega\\
}
\]
Hence in particular
\[
 i\delta_{X^Y} = i \tau \mathbin{\circ} \forall_Y \mathbin{\circ} (i\delta_X)^Y.
\]
\end{lem}

\begin{proof}
The right-hand square of the diagram commutes by uniqueness of classifying maps, while for the left-hand square we have $\tau i = 1$. Similarly, the bottom triangle commutes, because
\[
\xymatrix{
X^Y \ar[d]_{\Delta_{X^Y}} \ar[r]&1 \ar@{=}[r] \ar[d]^{\top^Y}& 1 \ar[d]^\top \\
X^Y \times X^Y \ar[r]_-{(\delta_X)^Y} &\Omega^Y \ar[r]_{\forall_Y} & \Omega\\
}
\] 
is a pullback diagram. (Note that the left-hand square is a pullback, because the functor $(-)^Y$, as a right adjoint, preserves these.)
\end{proof}

\begin{cor}
Modal function extensionality
\[
f: X^Y , g : X ^Y \mid \Box(\forall y : Y . f(y) =_X g(y)) \vdash f =_{X^Y} g.
\]
is true in any interpretation $(\tps{E},H)$.
\end{cor}

\begin{proof}
The left-hand side of the sequent is interpreted by the arrow 
\[
X^Y \times X^Y \xrightarrow{\lambda_Y (i\delta_X\langle ev \pi_{13}, ev \pi_{23} \rangle)}H^Y  \xrightarrow{\forall_Y} H \xrightarrow{\Box} H,
\]
where the projections come from $X^Y \times X^Y  \times Y$, and $ev :  X^Y  \times Y \rightarrow X$ is the canonical evaluation.
The right-hand side is simply 
\[
X^Y \times X^Y \xrightarrow{\delta_{X^Y}} \Omega_\tps{E} \xrightarrow{i} H.
\]
We need to show that the arrow 
\[
\langle i\tau \forall_Y\lambda_Y (i\delta_X\langle ev \pi_{13}, ev \pi_{23} \rangle) , i \delta_{X^Y} \rangle : X^Y \times X^Y \rightarrow H \times H
\]
factors through the partial ordering $(\leq) \rightarrowtail H \times H$.  Write the left-hand component as $i\varphi$.  It is enough to show that 
\[
\varphi \leq_{\Omega} \delta_{X^Y}  : X^Y \times X^Y,
\]
whence the claim follows as before, $i$ being order-preserving.

To show that the subobject $(Q,m)$ classified by the map $\tau \forall_Y\lambda_Y (i\delta_X\langle ev\pi_{13}, ev\pi_{23} \rangle)$ factors through $\Delta_{X^Y}$, as subobjects of $X^Y \times X^Y$, observe first that $\lambda_Y (i\delta_X\langle ev\pi_{13}, ev\pi_{23} \rangle)$ can be written as
\[
X^Y \times X^Y \xrightarrow{\eta} (X^Y \times X^Y \times Y)^Y \xrightarrow{\langle ev\pi_{13}, ev\pi_{23} \rangle^Y} (X \times X)^Y \xrightarrow{(\delta_X)^Y} \Omega^Y \xrightarrow{i^Y} H^Y,
\]
where $\eta$ is the unit component (at $X^Y \times X^Y$) of the product-exponential adjunction $(-) \times Y \dashv (-)^Y$. 
By the previous lemma 
\[
\tau \mathbin{\circ} \forall_Y \mathbin{\circ} i^Y \mathbin{\circ} (\delta_X)^Y = \delta_{X^Y}.
\] 
The subobject in question thus arises from pullbacks
\[
\xymatrix{
Q  \ar@{ >->}[d]_m \ar[rrr]&&& X^Y \ar[r] \ar[d]^{\Delta_{X^Y}} & 1 \ar[d]^{\top}\\%
X^Y \times X^Y \ar[rrr]^{\langle ev\pi_{13}, ev\pi_{23} \rangle^Y \mathbin{\circ} \eta} &&& X^Y \times X^Y \ar[r]^-{\delta_{X^Y}} & \Omega
}
\]
But $\langle ev\pi_{13}, ev\pi_{23} \rangle^Y \mathbin{\circ} \eta$ is the identity arrow. For it is the transpose (along the adjunction $(-) \times Y \dashv (-)^Y$) of 
\[
\langle ev\pi_{13}, ev\pi_{23} \rangle : X^Y \times X^Y \times Y \rightarrow X \times X.
\] 
The latter in turn is the canonical evaluation of $X^Y \times X^Y$ viewed as the exponential $(X \times X)^Y$, \ie the counit of the adjunction at $X \times X$, transposing which yields the identity.
As a result, 
\[
\tau \forall_Y\lambda_Y (i\delta_X\langle ev\pi_{13}, ev\pi_{23} \rangle) \leq_{\Omega} \delta_{X^Y},
\]
and therefore 
\[
i\tau \forall_Y\lambda_Y (i\delta_X\langle ev\pi_{13}, ev\pi_{23} \rangle) \leq_H i\delta_{X^Y}.
\qedhere
\]
\end{proof}

\begin{rem}
Before giving a counterexample to $i\delta_{X^Y} =   \forall_Y \mathbin{\circ} (i\delta_X)^Y$, let us remark that the equation does actually \emph{hold} in the topos \cat{Sets}. For consider $f\neq g \in X^Y$, \ie $f(y) \neq g(y)$, for some $y \in Y$. Then for any complete Heyting algebra $H$, the function $(i\delta_X)^Y(f,g) \in H^Y$ is defined as
\[
(i\delta_X)^Y(f,g)(y) = i\delta_X(f(y),g(y)) = \top, \quad \text{if}\ f(y) = g(y),
\]
and $\bot$ otherwise.
Thus taking the meet (cf.\ the definition in example \ref{Ex1}) yields 
\[
\bigwedge_{y \in Y} (i\delta_X)^Y(f,g)(y) = \bot,
\]
because $f(y) \neq g(y)$, for some $y \in Y$, by assumption. In turn the meet equals $\top$ just in case $f(y) = g(y)$, for all $y \in Y$, \ie if and only if $f = g$.
\end{rem}


\begin{prop}
It is not in general the case that for a topos $\tps{E}$ and a frame $H$ in~$\tps{E}$:
\[ 
i\delta_{X^Y} =   \forall_Y \mathbin{\circ} (i\delta_X)^Y.
\]
\end{prop}

\begin{proof}

To find a counterexample we consider a specific presheaf topos $\prs{C}$ described below.\foot{The counterexample, in particular the choice of \cat{C} and the functor $G: \cat{C} \rightarrow \cat{Sets}$ below, follows a slightly different, though equivalent, proof first given in \cite{joergsthesis}.}
Let's first recall some general facts. Write $\Omega_{|\cat{C}|}$ for the subobject classifier in $\cat{Sets}^{|\cat{C}|} $ and choose $H = f_\ast\Omega_{|\cat{C}|}$ (henceforth $\Omeag_\ast$), where $f$ is the geometric morphism $f : \cat{Sets}^{|\cat{C}|} \rightarrow \prs{C}$ induced by the inclusion $|\cat{C}| \hookrightarrow \cat{C}$ via right Kan extensions. Recall moreover from the beginning that the subobject classifier $\Omega$ of $\prs{C}$ determines for each $C$ the set of all sieves on $C$. By contrast, $\Omega_\ast(C)$ is the set of arbitrary sets of arrows with codomain $C$ (cf. also the example from the next section).

Recall that in any category of the form $\prs{C}$ the evaluation maps $\varepsilon : B^A \times A \rightarrow B$ have components
\[
\varepsilon_C(\eta , a) = \eta_C(1_C,a),
\]
where $\eta \in B^A(C) = \Hom{}{\y C \times A}{B}$ and $a \in A(C)$. The exponential transpose $\overline{\alpha} : Z \rightarrow B^A$ of a map $\alpha : Z \times A \rightarrow B$ has components
\begin{equation}\label{DefExTP}
\overline{\alpha}_C(z) = \alpha \mathbin{\circ} (\zeta \times 1_A),
\end{equation}
where $\zeta : \y C \rightarrow Z$ corresponds under the Yoneda lemma to the element $z \in Z(C)$, \ie is defined as $\zeta(f) = Z(f)(z)$, for any $f \in \y C(D)$.

For any object $A$ in \cat{C}, the functor $(-)^A$ acts on arrows $f : C \rightarrow D$ as 
\[
f^A = \overline{f \mathbin{\circ} \varepsilon},
\]
for evaluation $\varepsilon : C^A \times A \rightarrow C$. In particular, 
\[
(i\delta_B)^A  = \overline{i\delta_B \mathbin{\circ} \varepsilon},
\]
for $\varepsilon : (B \times B)^A \times A \rightarrow B \times B$ evaluation at $A$.
Thus, for any pair 
\[
\langle \eta,\mu \rangle \in (B \times B)^A(C) = \Hom{}{\y C \times A}{B \times B},
\]
we have 
\[
(\overline{i\delta_B \mathbin{\circ} \varepsilon})_C(\eta , \mu) = i \delta_B \varepsilon ( \langle \eta , \mu \rangle^\ast \times 1_A) = i \delta_B\langle \eta , \mu \rangle.
\]
Here we use that $\langle \eta , \mu \rangle^\ast : \y C \rightarrow (B \times B)^A$ corresponds under Yoneda to the element $\langle \eta , \mu \rangle \in (B \times B)^A(C) = \Hom{}{\y C \times A}{B \times B}$ and that $\langle \eta , \mu \rangle^\ast$ is equal to the exponential transpose of $\langle \eta , \mu \rangle$.
Accordingly,
\begin{align*}
\forall_C(i\delta_B)^A_C(\eta,\mu)  & = \forall_C(\overline{i\delta_B \mathbin{\circ} \varepsilon})_C(\eta , \mu)\\
& = \forall_C(i \delta_B\langle \eta , \mu \rangle) \\
& = \bigcup \{ s \in \Omega_\ast(C) \mid \Omega_\ast(g)(s) \leq i_D (\delta_B)_D (\eta_D(g,b) , \mu_D(g,b)), \mtext{for all}  \\ 
& \phantom{= \bigcup \{ s \in \Omega_\ast(C) \mid }\ (g: D \rightarrow C , b \in A(D))\},
\end{align*}

On the other hand, the classifying map of the diagonal on a functor $B : \cat{C}^{op} \rightarrow \cat{Sets}$ is computed as
\[
(\delta_B)_C(x,y) = \{ f : D \rightarrow C \mid B(f)(x) = B(f)(y)\},
\] 
for all pairs $(x,y) \in B(C)\times B(C)$. It is the maximal sieve $\top_C$ on $C$ just in case $x = y$.

Now let $\cat{C}$ be the finite category 
\[
C \xrightarrow{g} D,
\]
and define a functor $G : \cat{C}^{op} \rightarrow \cat{Sets}$ as follows:\foot{Although $g: C \rightarrow D$ may be seen as the two-element poset with resulting presheaf topos $\cat{Sets}^\rightarrow$, we will not need that description. The objects and arrows in \cat{C} merely play the role of indices, so it seems better to use the more neutral notation $C,D,g$.}
\[
G(D) = \{u\},\ G(C) = \{v,w\},\ G(g)(u) = v.
\]
Furthermore, choose $\eta , \mu \in G^G(D)$ such that $\eta \neq \mu$. Observe that, while necessarily 
\[
\eta_D = \mu_D : \y D (D) \times G(D) \rightarrow G(D)
\]
with assignment 
\[
(1_D , u) \mapsto u,
\]
we can chose $\eta,\mu$  in such a way that $\eta_C(g,x) \neq \mu_C(g,x)$, for some pair $(g,x) \in \y D (C) \times G(C)$. Specifically, since the first component $g$ is fixed, the choice is only about $x \in G(C)$ which in turn must concern $w \in G(C)$. For naturality requires that 
\[
G(g)\eta_D(1_D,u) = \eta_C(\y D(g) \times G(g))_C(1_D,u) = \eta_C(g,v),
\] 
so that since $G(g)\eta_D(1_D,u) = G(g)(u) = v$, we must have $\eta_C(g,v) = v$; similarly $\mu_C(g,v) = v$. However, no constraint is put on the values $\eta_C(g,w)$ and $\mu_C(g,w)$, respectively.

Then:
\begin{equation}\label{CExFunExt1}
(\delta_{G^G})_D(\eta, \mu) = \{x : X \rightarrow D \mid G^G(x)(\eta) = G^G(x)(\mu)\} = \emptyset .
\end{equation}
For if $x = g$, observe
\[
G^G(g)(\eta) = \eta \mathbin{\circ} (\y g \times 1_G) \neq \mu \mathbin{\circ} (\y g \times 1_G) = G^G(g)(\mu),
\]
because 
\[
\eta_C (\y g \times 1_G)_C(1_C,w) = \eta_C(g,w)   \neq \mu_C(g,w) = \mu_C(\y g \times 1_G)_C(1_C,w), 
\]
where the inequality holds by construction. But also, if $x=1_D$, then $G^G (x)(\eta) = \eta \neq \mu = G^G(x)(\mu)$, where the inequality holds by assumption again.

On the other hand, 
\begin{equation}\label{ExFunExt1}
\forall_D(i\delta_G)^G_D(\eta,\mu)  = \bigcup \{ s \in \Omega_\ast(D) \mid \Omega_\ast(x)(s) \leq i_X (\delta_G)_X (\eta_X(x,b) , \mu_X(x,b))\} = \{1_D\}.
\end{equation}
for all pairs $(x : X \rightarrow D, b \in G(X))$ from \cat{C}.
It is clear that $s = \{1_D\}$ satisfies the condition on the underlying set of the union, since for $x = 1_D$,
\begin{align*}
\Omega_\ast(1_D)(\{1_D\}) & = \{1_D\} \\
& \subseteq \top_D = i_D(\delta_G)_D (\eta_D(1_D , u) , \mu_D(1_D , u)).
\end{align*}
On the other hand, for $x = g$, it is trivially always the case that
\[
\Omega_\ast(g)(\{1_D\}) = \emptyset \subseteq (\delta_G)_C( \eta_C(g,b) , \mu_C(g,b)),
\] 
for all $b \in G(C)$.

Furthermore, note that if $g \in s$, for some $s \in \Omega_\ast(D)$, then 
\[
\Omega_\ast(g)(s) = \top_C = \{1_C\}.
\]
So if $g \in s$, for some $s$ in the underlying set of the union $\eqref{ExFunExt1}$, we had to have
\[
\top_C = \Omega_\ast(g)(s) \leq i_C (\delta_G)_C (\eta_C(g,b) , \mu_C(g,b)),
\]
for all $b \in G(C)$. However, since by assumption $\eta_C(g,w) \neq \mu_C(g,w)$, 
\[
(\delta_G)_C( \eta_C(g,w) , \mu_C(g,w)) = \emptyset,
\]
and so
\[
\Omega_\ast(g)(s) \nleq i_C (\delta_G)_C (\eta_C(g,w) , \mu_C(g,w)).
\]
Thus $g \notin s$, for all $s \in \Omega_\ast(D)$ in the underlying set of $\forall_D(i\delta_G)^G_D(\eta,\mu)$. Therefore
\[
\forall_D(i\delta_G)^G_D(\eta,\mu) = \{1_D\},
\]
as claimed, and in contrast to \eqref{CExFunExt1}:
\[
i_D(\delta_{G^G})_D(\eta, \mu) = \emptyset.
\]
(Of course, $\tau(\{1_D\}) = \emptyset$, as lemma \ref{FunExtLem} predicts.)
\end{proof}


\begin{rem}
There is an alternative, more combinatorial way of presenting the previous proof. The idea is to formulate the proof in terms of loop graphs rather than presheaves. For presheaves on the category $\{C \xrightarrow{g} D\}$ can equivalently be regarded as labelled graphs that consist only of loops and points, for instance:
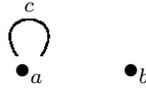
\[
\xymatrix{
&\\
\bullet_a \ar@(ur,ul)@{-}[]_c & \bullet_b \\
}
\]
Here, $G(D)$ is the set of edges and $G(C)$ the set of vertices, while $G(g)$ assigns to an edge a point, its ``source''. Thus every loop has a unique source but each point may admit several edges on it. $\Omega$ is the following graph which is easily seen to classify subgraphs:
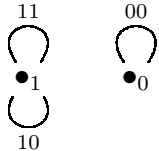
\[
\xymatrix{
&\\
\bullet_1 \ar@(ur,ul)@{-}[]_{11} \ar@(dr,dl)@{-}[]^{10} & \bullet_0  \ar@(ur,ul)@{-}[]_{00} \\
}
\]
The labelling expresses the imposed algebraic structure of $\Omega$ with $0 < 1$ and $xy \leq uv \mathiff x \leq u\ \&\ y \leq v$. Intuitively, in presheaf terms, $1$ stands for the maximal sieve on $C$ and $0$ for the empty sieve; similarly pairs $xy$ encode sieves on $D$, where $x = 1$ if and only if $g$ is the sieve and $y=1$ if and only if $1_D$ is in it. Then the source of an edge $xy$ is just $x$. For instance, the sieve $\{g\}$ on $D$ is encoded by $10$. Then $\Omega(g)(\{g\}) = \{1_C\}$ which is encoded by $1$. Note also that the set of edges is the three-element Heyting algebra from example \ref{FrameEx}.

By contrast $\Omega_\ast$ is the graph
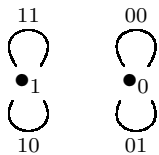
\[
\xymatrix{
&\\
\bullet_1 \ar@(ur,ul)@{-}[]_{11} \ar@(dr,dl)@{-}[]^{10} & \bullet_0  \ar@(ur,ul)@{-}[]_{00} \ar@(dr,dl)@{-}[]^{01}\\
}
\]
Here the additional edge $01$ corresponds to the fact that $\{1_D\} \in \Omega_\ast(D)$. Thus the set of edges is the four-element Boolean algebra with the source map $2^2 \rightarrow 2$ induced by the inclusion $1 \hookrightarrow 2$.

The functor $G$ from before becomes the graph
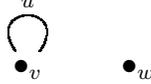
\[
\xymatrix{
&\\
\bullet_v \ar@(ur,ul)@{-}[]_u & \bullet_w
}
\]
while $G^G$ is 
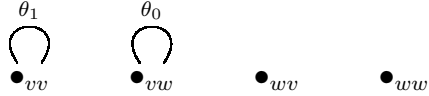
\[
\xymatrix{
\bullet_{vv} \ar@(ur,ul)@{}[]_{\theta_1} &\bullet_{vw} \ar@(ur,ul)@{}[]_{\theta_0} & \bullet_{wv} & \bullet_{ww}
}
\]
The graph $\Omega^G$ then looks like this:
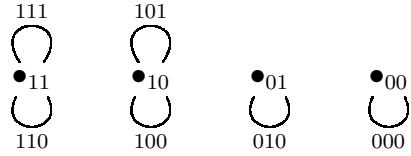
\[
\xymatrix{ 
&&&\\
\bullet_{11} \ar@(ur,ul)@{-}[]_{111} \ar@(dr,dl)@{-}[]^{110} & 
\bullet_{10} \ar@(ur,ul)@{-}[]_{101} \ar@(dr,dl)@{-}[]^{100} & 
\bullet_{01} 
 \ar@(dr,dl)@{-}[]^{010} & 
\bullet_{00} \ar@(dr,dl)@{-}[]^{000} 
}
\]
again with the pointwise ordering.\foot{The labelling can of course systematically be translated into one such that \eg edges are labelled by natural transformations $\eta : \y D \times G  \rightarrow \Omega$ as before. For any such $\eta$ is uniquely determined by the values $\eta_D(1_D,u)$ and $\eta_C(g,w)$. Vertices are just $2^2$, as there are exactly four natural transformations $\y C \times G \rightarrow \Omega$, each one defined by the pair $xy$ of values of the component at $a$ ($\Omega(C) = 2$). Their intuitive meaning in terms of sieves on $D$ is as before. In turn, the notation $xyz$ is chosen in such a way that the source is $xy$. Thus, $xyz$ is to be read so as to mean $\eta_D(1_D,u) = xz$ and $\eta_C(g,w) = y$. For by definition the source of an edge $\eta$ in $\Omega^G$ is $\Omega^G(g)(\eta)= \eta(\y g \times 1_G)$. Its component at $D$ is empty while for $C$, and $x = v$ 
\[
\eta_C((\y g)_C \times 1_{GC})(1_C,v) = \eta_C(g,v) = \eta_C(\y D(g) \times G(g))(1_D,u) = g^\ast \eta_D(1_D,u),
\]
where the last identity holds by naturality of $\eta$. Thus the source is the pair $(g^\ast\eta_D(1_D,u) , \eta_C(g,w))$. In turn, $g^\ast\eta_D(1_D,u)$ is the first digit of $\eta_D(1_D,u)$.  Moreover, in the expression $xyz$, $y = 1 \mathiff 1_C \in \eta_C(g,w)$. So the source of $xyz$ is $xy$.
}

The graph $\Omega_\ast^G$ is:
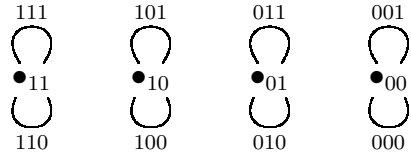
\[
\xymatrix{ 
&&&\\
\bullet_{11} \ar@(ur,ul)@{-}[]_{111} \ar@(dr,dl)@{-}[]^{110} & 
\bullet_{10} \ar@(ur,ul)@{-}[]_{101} \ar@(dr,dl)@{-}[]^{100} & 
\bullet_{01} \ar@(ur,ul)@{-}[]_{011} \ar@(dr,dl)@{-}[]^{010} & 
\bullet_{00} \ar@(dr,dl)@{-}[]^{000} \ar@(ur,ul)@{-}[]_{001}
}
\]
The vertices are the four element Boolean algebra $2^2$ with the pointwise ordering, and the same for the edges $2^3$. The source map $xyz \mapsto xy$ is the map $2^3 \rightarrow 2^2$ induced by the inclusion $2 \hookrightarrow 3$ that projects out the first two arguments of an element of $2^3$. 

As it turns out, for $\delta^G : (G \times G)^G \rightarrow \Omega^G$:
\[
(\delta^G)_D (\theta_0,\theta_1) = 101.
\]
On the other hand,  $\Delta_C(x) = xx$ and $\Delta_D(xy) = xxy$, and so 
\[
\forall_D(xyz) = \bigvee \{st \in \Omega_\ast(D) \mid sst \leq xyz\},
\]
and similarly for $\Omega$. Thus $\forall_D(101) = \bigvee\{00 , 01\} = 01$, for $\forall_D : \Omega_\ast^G(D) \rightarrow \Omega_\ast(D)$, while  $\forall_D(101) = \bigvee\{00\} = 00$, for $\forall_D : \Omega^G(D) \rightarrow \Omega(D)$.

Note finally that function extensionality \emph{is} valid in constant domain models. (See next section for the connection between topos semantics and Kripke models.)
For instance, consider a loop graph where $G(D) \cong 2 \cong G(C)$. An element in $\Omega^G(D)$, as a natural transformation $\eta_D : \y D \times G \rightarrow \Omega$, is completely determined by the two values $\eta_D(1,a), \eta_D(1,b)$, for $\{a,b\} = G(D)$. Thus, edges in $\Omega^G$ can be represented by sequences $xyzw$, where $xy$ and $zw$ are the respective edges $\eta_D(1,a)$ and  $\eta_D(1,b)$ in $\Omega(D)$, using the binary notation from before. The source of an edge $xyzw$ is $xz$. On the other hand, the map $\Delta_D : \Omega(D) \rightarrow \Omega^G(D)$ can be computed as $\Delta_D(st) = stst$. Now note that there can be no edge in $\Omega^G$ of the form $xy01$ or $01zw$, because $01$ is not an edge in $\Omega$ (moreover that's the only difference between $\Omega^G$ and $\Omega_\ast^G$). As a result, there is no edge in $\Omega^G$ such that applying $\forall$ to it is different from applying $\forall$ to that same edge in $\Omega_\ast^G$. For the only reason this might happen is because $01$ is in the underlying set of the join
\[
\forall_D(xyzw) = \bigvee \{st \in \Omega_\ast(D) \mid stst \leq xyzw\}.
\]
However, if $0101 \leq xyzw$, for any edge $xyzw$ in $\Omega^G$, then $xyzw = 1111$. But certainly $\forall$ has the same value on $1111$ for both $\Omega^G$ and $\Omega_\ast^G$. Although the argument is for models with domain of cardinality 2, it easily generalizes to any $n$.

\end{rem}

\normalsize


\section{Algebraic semantics from geometric morphisms}\label{RecGeomCase}

The canonical example of a model in the sense of def.\ \ref{DefMod} is the case where $H = \gsoc{f}{F}$, for a surjective geometric morphism $f : \tps{F} \rightarrow \tps{E}$ \cite{ghilardi07,makkaireyes95,reyes91,reyeszolfaghari91}. We will continue to describe it in some detail to show that the known semantics for it really coincides with the one described in section \ref{HOMLModels}, the crucial thing to check being the equality relation. To ease notation, we write $A^\ast$ for $f^\ast A$, $A_\ast$ for $f_\ast A$ and $\Omega_\ast$ for $\gsoc{f}{F}$, if $f$ is understood.

\begin{prop}\label{PropRecGeo}
For any geometric morphism $f : \tps{F} \rightarrow \tps{E}$, the object $\Omega_\ast$ is a complete Heyting algebra in $\tps{E}$.
\end{prop}

\begin{proof}

The object $\Omega_\ast$ is a Heyting algebra under the image of $f_\ast$, since $f_\ast$ preserves products. The same algebraic structure is equivalently determined through Yoneda by the external Heyting operations on each $\sub{\tps{F}}{A^\ast}$ under the natural isomorphisms
\[
\sub{\tps{F}}{A^\ast} \cong \Hom{\tps{F}}{A^\ast }{\soc{F}} \cong \Hom{\tps{E}}{ A}{\Omega_\ast}.
\]

Completeness means that $\Omega_\ast$ has $I$-indexed joins and meets, for any object $I$ in $\tps{E}$. One way to see this is to first note that there are isomorphisms (natural in $E$)
\[
\Hom{}{E}{(\Omega_\ast)^I} \cong \Hom{}{E \times I}{\Omega_\ast} \cong \Hom{}{E^\ast  \times I^\ast}{\Omega_{\tps{F}}} \cong \Hom{}{E^\ast}{\Omega_{\tps{F}}^{I^\ast }},
\]
where we use that $f^\ast$ preserves finite limits.
Composition with 
\[
{\bigmeet_{I^\ast }} : \Omega_{\tps{F}}^{I^\ast } \longrightarrow \Omega_{\tps{F}}
\]
hence yields a function 
\[
\Hom{}{E}{(\Omega_\ast)^I} \xrightarrow{\cong} \Hom{}{E^\ast }{\Omega_{\tps{F}}^{I^\ast }} \xrightarrow{\bigmeet_{}{}_{I^\ast } \mathbin{\circ} (-)}  \Hom{}{E^\ast }{\Omega_{\tps{F}}} \xrightarrow{\cong} \Hom{}{E}{\Omega_\ast},
\]
all natural in $E$.
Thus, by the Yoneda lemma, there is a unique map 
\[
\bigmeet_I : (\Omega_\ast)^I \longrightarrow \Omega_\ast
\]
such that the function 
\[
\Hom{}{E}{(\Omega_\ast)^I} \longrightarrow \Hom{}{E}{\Omega_\ast}
\]
from above is induced by composition with $\bigmeet_I$. 

$\forall_I$ is indeed right adjoint to $\Delta_I : \Omega_\ast \rightarrow \Omega_\ast^I$. For $\Delta_{I^\ast}: \soc{F} \rightarrow \soc{E}^{I^\ast}$ induces, by composition, a function
\[
\Hom{}{E}{\Omega_\ast} \cong  \Hom{}{E^\ast }{\Omega_{\tps{F}}} \ \xrightarrow{\Delta_{I^\ast} \mathbin{\circ} (-)} \Hom{}{E^\ast }{\Omega_{\tps{F}}^{I^\ast }} 
\]
with
\[
 \Delta_{I^\ast } \mathbin{\circ} (-) \dashv \bigmeet_{}{}_{I^\ast } \mathbin{\circ} (-).
\]
This adjunction in turn is the one that corresponds by Yoneda under the isomorphism \eqref{DefSOC} to the adjunction $\pi_1^\ast \dashv \forall_{\pi_1}$:
\[
\forall_{\pi_1} : \sub{\tps{F}}{E^\ast \times I^\ast} \leftrightarrows \sub{\tps{E}}{E^\ast} : \pi_1^\ast ,
\]
where $\pi_1^\ast$ is pulling back along $\pi_1 : E^\ast \times I^\ast \rightarrow E^\ast$. $I$-indexed joins are treated similarly.
\end{proof}

The modal operator is given by the uniquely determined structure 
\begin{equation}\label{IntAd}
\tau : \Omega_\ast \leftrightarrows \soc{E} : i,
\end{equation}
where $\tau$ is the classifying map of 
\[
\top = f_\ast (\top) : 1 \rightarrow \Omega_\ast.
\]

\begin{lem}\label{ExtIntAdj}
The internal adjunction \eqref{IntAd} is induced via the Yoneda lemma by an external adjunction 
\begin{equation}\label{ExtModAd}
\Delta_A :\sub{\tps{E}}{A} \leftrightarrows \sub{\tps{F}}{f^\ast A} : \Gamma_A
\end{equation}
which is natural in $A$.\foot{Cf. \eg \cite{reyeszolfaghari91}.}
\end{lem}

\begin{proof}
Here, $\Delta_A$ is $f^\ast$ restricted to subobjects of $A$. It follows that $\Delta_A$ is an injective frame map, as $f^\ast$ is a faithful left exact left adjoint.
On the other hand, $\Gamma_A(X,m)$, for any mono $m : X \rightarrowtail f^\ast A$, is by definition the left-hand map in the following pullback
\[
\xymatrix{
\bullet \ar[r] \ar@{ >->}[d] & f_\ast X \ar@{ >->}[d]^{f_\ast m} \\
A \ar[r]_-{\eta_A} & f_\ast f^\ast A.
}
\]
The resulting two functions, natural in $A$, have the form:
\begin{equation}\label{IsoExtMod}
\Hom{\tps{E}}{A}{\soc{E}} \cong \sub{\tps{E}}{A} \leftrightarrows \sub{\tps{F}}{f^\ast A} \cong \Hom{\tps{F}}{f^\ast A}{\soc{F}} \cong \Hom{\tps{F}}{A}{\Omega_\ast}.
\end{equation}
By Yoneda they determine maps 
\[
\delta : \soc{E} \leftrightarrows  \Omega_\ast : \gamma,
\]
internally adjoint given that $\Delta_A \dashv \Gamma_A$, for each $A$ in $\tps{E}$. The map $\delta$ is monic, because each $\Delta_A$ is injective. It readily follows that $\delta = i$ and $\gamma = \tau$. For $\delta$ is a monic frame map and $\delta \dashv \gamma$, while the arrow 
\[
\gamma : \Omega_\ast \longrightarrow \soc{E}
\] 
obtained through the Yoneda lemma as above actually is the classifying map of the top element $f_\ast\top : 1 \rightarrow \Omega_\ast$.  
\end{proof}

\begin{lem}
The internal structure $\Omega_\ast$ is a faithful frame, \ie the canonical frame map $i : \soc{E} \rightarrow \Omega_\ast$ is a monomorphism.
\end{lem}

\begin{proof}
Since the maps $\Delta_A$ in lemma \ref{ExtIntAdj} are injective, this means that $\Delta : \sub{\tps{E}}{-} \rightarrow \sub{\tps{F}}{f^\ast (-)}$ is a monic natural transformation. As $i: \soc{E} \rightarrow \Omega_\ast$ is obtained using the Yoneda lemma from the maps $\Delta_A$, it readily follows that $i$ is monic, because the Yoneda embedding reflects monomorphisms. 
\end{proof}

Formulas $\varphi$ (in one free variable, say) are thus interpreted equivalently in any of the following ways (let $M$ interpret the type of $x$):
\[
\sem{\varphi} \in \sub{\tps{F}}{f^\ast  M}, \ \ \ \ \ \  M^\ast  \xrightarrow{\sem{\varphi}} \soc{F}, \ \ \ \ \ \  M \xrightarrow{\sem{\varphi}} \Omega_\ast, 
\]
where the third one follows from definition \ref{DefMod}.

Moreover, let $\delta_{M^\ast}$ be the classifying map of the diagonal $\langle 1_{M^\ast} , 1_{M^\ast} \rangle : M^\ast \to M^\ast \times M^\ast$. We will write its transpose along $f^\ast \dashv f_\ast$ simply as

\begin{equation}\label{deltastar}
M \times M \xrightarrow{\delta_\ast} \Omega_\ast
\end{equation}
when $M$ is clear. Then we have:

\begin{lem}
The equality predicate for $M^*$ may  be interpreted by the map \eqref{deltastar}, obtained as the transpose along $f^\ast \dashv f_\ast$ of 
\[
(M \times M)^\ast \cong M^\ast  \times M^\ast  \xrightarrow{\delta_{M^\ast}} \soc{F}.
\]
\end{lem}
\noindent The proof is immediate, given the soundness of the interpretation with respect to~$\delta_{M^\ast}$.

\begin{mydef}
By a \emph{geometric model} we shall mean a model derived from a geometric morphism in this way; specifically, where $f : \tps{F} \to \tps{E}$ and $H=f_*(\Omega_{\tps{F}})$.
\end{mydef}


To show, finally, that  geometric models are a special case of  algebraic ones, the main thing that needs to be verified is that equality is interpreted the same way in each case, i.e.:
\[
\delta_\ast = i \mathbin{\circ} \delta_M.
\]

First, we make the following observation:

\begin{lem}\label{BoxLem1}
For any map $\alpha : D \rightarrow \Omega_\ast$, we have $i \tau \mathbin{\circ} \alpha = \alpha$ iff the subobject classified by the transpose $\widetilde{\alpha} : f^\ast D \rightarrow \soc{F}$ of $\alpha$ is of the form $f^\ast m : f^\ast A \rightarrowtail f^\ast D$, for some $m : A \rightarrowtail D$ in $\tps{E}$. \hfill $\Box$
\end{lem}

\begin{prop}
For any object $D$ in $\tps{E}$, and any geometric morphism $f : \tps{F} \rightarrow \tps{E}$:
\[
\delta_\ast = i \mathbin{\circ} \delta_D.
\]
\end{prop}

\begin{proof}
We prove this by showing 
\[
\tau \mathbin{\circ} \delta_\ast  = \delta_D,
\]
whence the statement follows from $\delta_\ast = i \mathbin{\circ} \tau \mathbin{\circ} \delta_\ast  = i \mathbin{\circ} \delta_D$, where the identity $\delta_\ast = i \mathbin{\circ} \tau \mathbin{\circ} \delta_\ast$ holds by applying  \lemref{BoxLem1} to $\delta_\ast$. 

The proof is essentially contained in the following diagram
\[
\xymatrix{
D \ar[rr]^{\eta_D} \ar[dd]_{\Delta_D}  && (D^\ast)_\ast \ar[rr] \ar[dd]|{(\Delta^\ast)_\ast} && 1\ar[dd]^\top \ar@{=}[drr]&\\
&&&&&&1\ar[dd]^\top\\
D \times D \ar[rr]^-{\eta_D \times \eta_D} \ar@/_1.7pc/[rrrr]_{\delta_\ast} \ar@/_1.7pc/[drrrrrr]_{\delta_D} && (D^\ast)_\ast \times (D^\ast)_\ast \ar[rr]^-{(\delta_{D^\ast})_\ast} && \Omega_\ast \ar[drr]_\tau &\\
&&&&&&\Omega_{\tps{E}}\\ 
}
\]
where $\Delta_D = \langle 1_D , 1_D \rangle$, $\eta$ is the unit of $f^\ast \dashv f_\ast$, and $\delta, \tau$ denote the respective classifying maps. 
The square in the middle is a pullback, since $f_\ast$ preserves them. Moreover, by the definition of $\delta_D$, the large outer square is a pullback. Note further that $\delta_\ast = (\delta_{D^\ast})_\ast \mathbin{\circ} \eta_{D\times D}$, by the definition of $\delta_\ast$ as the transpose of $\delta_{D^\ast}$ along $f^\ast \dashv f_\ast$. Thus the desired equality would follow if the unit square were a pullback, for then 
\[
\tau \mathbin{\circ} (\delta_{D^\ast})_\ast \mathbin{\circ} \eta_{D\times D} = \tau \mathbin{\circ} \delta_\ast 
\]
 would classify  $\Delta_D$, and so $\tau \mathbin{\circ} \delta_\ast  = \delta_D$.  This is in fact the case. For $f : \tps{F} \rightarrow \tps{E}$ being surjective (\ie $f^\ast$ faithful) implies that the unit components, and therefore $\eta_D \times \eta_D$, are monic. A direct verification then shows that the square is a pullback.
\end{proof}




\begin{exam} \emph{Kripke Models.} 
As is well known, any functor $F : \cat{C} \rightarrow \cat{D}$ induces a geometric morphism
\[
f^\ast\dashv f_\ast : \srp{C} \rightarrow \srp{D} ,
\]
where $f^\ast$ is precomposition with $F$, and $f_\ast$ is a right Kan extension. Let $\cat{C} = |\cat{D}|$ and $F$ the inclusion $i: |\cat{D}| \rightarrow\cat{D}$. Then the induced geometric morphism $i^\ast \dashv i_\ast : \cat{Sets}^{|\cat{D}|} \rightarrow \srp{D}$ is surjective. The subobject classifier $\Omega_\cat{D}$ in $\srp{D}$ consists, for each $D$, of the set of cosieves on $D$, which can be construed as the functor category 
\[
2^{D/\cat{D}},
\]
where $2$ is viewed as the poset $\{0 \leq 1\}$; while $\Omega_{|\cat{D}|}(D) = 2$, for each $D$ in \cat{D}.  


On the other hand, by the definition of right Kan extension, $i_\ast\Omega_{|\cat{D}|}(D) = \prod_{h \in D/\cat{D}} 2 = 2^{|D/\cat{D}|}$, as can also be seen from 
\[
i_\ast\Omega_{|\cat{D}|}(D) \cong \Hom{\widehat{\cat{D}}}{\y D}{i_\ast\Omega_{|\cat{D}|}} \cong \Hom{\widehat{|\cat{D}|}}{i^\ast (\y D)}{\Omega_{|\cat{D}|}}.
\]
The last set is (isomorphic to) the set of subfamilies of the functor $i^\ast (\y D) : |\cat{D}| \rightarrow \cat{Sets}$, by the definition of the subobject classifier $\Omega_{|\cat{D}|}$: each natural transformation
\[
i^\ast\y D = \y D\mathbin{\circ} i = \Hom{\cat{D}}{D}{-} \longrightarrow 2
\]
determines, for each $D'$ in \cat{D}, a set of arrows $D \rightarrow D'$. On arrows $h : D \rightarrow D''$, the functor $i_\ast\Omega_{|\cat{D}|}$ is the function $i_\ast\Omega_{|\cat{D}|}(h) : i_\ast\Omega_{|\cat{D}|}(D) \rightarrow i_\ast\Omega_{|\cat{D}|}(D'')$ defined as
\[
i_\ast\Omega_{|\cat{D}|}(h)(A) = \{f : D'' \rightarrow X \mid f \mathbin{\circ} h \in A\}.
\]
The components of the (internal) adjunction $i : \Omega_\cat{D} \leftrightarrows i_\ast\Omega_{|\cat{D}|} : \tau$ then read
\[
i_D : 2^{D/\cat{D}} \leftrightarrows 2^{|D/\cat{D}|} : \tau_D,
\]
where $i_D \dashv \tau_D$ ``externally''. It is not hard to see that $i$ is the inclusion, while 
\[
\tau_D(A) = \bigvee \{S \in 2^{D/\cat{D}} \mid i_D(S) \leq A\},
\]
by the definition of right adjoint to the frame map $i$ (cf. \eqref{TAdFM}). In words, $\tau$ maps any family of arrows with domain $D$ to the largest cosieve on $D$ contained in it.
In particular, when $\cat{D}$ is a preorder, then $D/\cat{D} = \uparrow (D)$, the upward closure of $D$; while $2^{D/\cat{D}}$ is the set of all monotone maps $\uparrow(D) \rightarrow 2$, \ie upsets of $\uparrow (D)$, while $2^{|D/\cat{D}|}$ is the set of arbitrary subsets of $\uparrow (D)$.

An arrow $\varphi : E \rightarrow i_\ast\Omega_{|\cat{D}|} = 2^{|-/\cat{D}|}$ in $\srp{D}$ defines an indexed subfamily $P$ of the functor $F$, and conversely. Explicitly, given such  $\varphi : E \rightarrow i_\ast\Omega_{|\cat{D}|}$, define subsets $P_\varphi(D) \subseteq E(D)$, for each $D$ in \cat{D} and $a \in E(D)$, by
\begin{equation}\label{KMods1}
a \in P_\varphi(D) \mathiff 1_D \in \varphi_D(a).
\end{equation}
Conversely, given maps $E(D) \rightarrow 2$, \ie components of an arrow $i^\ast E \rightarrow \Omega_{|\cat{D}|}$ in $\cat{Sets}^{|\cat{D}|}$, or equivalently a subfamily $P$ of $E$, define a natural transformation $\varphi_P : E \rightarrow i_\ast\Omega_{|\cat{D}|}$ by 
\begin{equation}\label{KMods2}
(\varphi_P)_D(a) = \{f : D \rightarrow C \mid E(f)(a) \in P(C) \},
\end{equation}
These constructions are mutually inverse and so describe the canonical isomorphism
\[
\Hom{}{E}{i_\ast\Omega_{|\cat{D}|}} \cong \Hom{}{i^\ast E}{\Omega_{|\cat{D}|}}  \cong \sub{}{i^\ast E}.
\]
Note also that the transpose $\overline{\varphi} = \varepsilon\varphi^\ast$ of $\varphi : E \rightarrow \Omega_\ast$ along the adjunction $f^\ast \dashv f_\ast$ actually is the classifying map in $\cat{Sets}^{|\cat{D}|}$ of the subobject $P_\varphi$ of $f^\ast E$ defined in \eqref{KMods1}:
\begin{align*}
\varepsilon_C\varphi^\ast_C(a) = 1 & \mathiff 1_C \in \varphi^\ast_C(a)\\
& \mathiff 1_C \in \varphi_C(a) \\
& \mathiff a \in P_\varphi(C),
\end{align*}
for any $a \in E(C)$.

On the other hand, considering $\Omega_{\cat{D}} = 2^{D/\cat{D}}$ instead of $2^{|D/\cat{D}|}$, the same definitions \eqref{KMods1} and \eqref{KMods2} establish a correspondence between sub\emph{functors} of $E$ and their classifying maps in $\srp{D}$.  In particular, the classifying map of a subfunctor of $E$ factors through $i_\ast\Omega_{|K|}$ via $\tau$.

Thus, when \cat{D} is a preorder,  algebraic models in the complete Heyting algebra $i_\ast \Omega_{|K|}$ are precisely Kripke models on \cat{D}. The ``domain'' of the model is given by the functor $E$, while each $E(D)$ is the domain of individuals at each world $D$. Each formula determines, as an arrow $\varphi : E \rightarrow i_\ast \Omega_{|K|}$, a subfamily of $E$, that is a family $(P_\varphi(D) \subseteq E(D))$. Then $\tau$ determines the largest compatible subfamily of that family, \ie a family closed under the action of $E$. Indeed, for $x \in E(D)$, 
\[
x \in P_{\tau\varphi}(D) \mathiff 1_D \in (\tau\varphi)_D(x).
\]
Now $(\tau\varphi)_D(x)$ is the maximal sieve on $D$ just in case $\varphi_D(x)$ is. So, if satisfied, the right-hand side means that $x \in P_\varphi(D)$ and moreover $F(f)(x) \in P_\varphi(C)$, for all $C \geq D$. Semantically speaking, $x$ satisfies $\tau\varphi$ (at $D$) just in case $x$ (or rather its ``counterpart'' $F_{CD}(x)$) satisfies $\varphi$ in all worlds accessible from $D$.

Thus we recovered the natural adjunction 
\[
\Delta_E : \sub{}{E} \leftrightarrows \sub{}{i^\ast E} : \Gamma_E
\]
that succinctly describes the algebraic structure of Kripke models.

Lastly, presheaf semantics reduces to standard Kripke semantics for propositional modal logic in the following sense. In the latter, propositional formulas are recursively assigned elements in $\pow{\cat{K}}$, for a preorder $\cat{K}$. Let $\pow{\downarrow(-)} = \Omega_\ast$ be the composite functor 
\[
\cat{K} \xrightarrow{\downarrow} \cat{Sets} \xrightarrow{\pow{-}} \cat{Sets}^{op}.
\]
Observe that 
\[
\pow{\cat{K}} \cong \Hom{\prs{K}}{1}{\pow{\downarrow(-)}},
\]
via assignments (where $\varphi \subseteq \pow{\cat{K}}$)
\[
\varphi \mapsto (\varphi_k =\ \downarrow(k \cap \varphi) \mid k \in \cat{K})
\]
and 
\[
(\varphi_k \mid k \in \cat{K}) \mapsto \bigcup_k \varphi_k.
\]

Thus modelling formulas (in one variable, say) by maps of presheaves 
\[
M \longrightarrow \pow{\downarrow(-)} = \Omega_\ast
\] 
yields precisely the familiar Kripke model idea for propositions, \ie closed formulas. Moreover, for constant domains:
\[
\Hom{\prs{K}}{\Delta M}{\pow{\downarrow(-)}} \cong \Hom{\cat{Sets}}{M}{\varprojlim\pow{\downarrow(-)}} \cong \Hom{\cat{Sets}}{M}{\pow{K}}.
\]
Here, $\Delta : \cat{Sets} \longrightarrow \prs{K}$ is the functor $\Delta(M)(k) = M$, for any set $M$ and $k \in \cat{K}$.
A function $\varphi : M \longrightarrow \pow{K}$  assigns to each individual in the domain $M$ a set of worlds for which the individual satisfies the formula represented by $\varphi$.

\end{exam}

\paragraph{Kripke-Joyal forcing:}
Another way of seeing the close relation between presheaf semantics and Kripke semantics is via the notion of ``Kripke-Joyal forcing'' \cite{maclanemoerdijk92, lambekscott88}. For any topos $\tps{E}$ one can define a forcing relation $\Vdash$ to interpret intuitionistic higher-order logic .
Given an arrow $\varphi : M \rightarrow \soc{E}$,  let $S_\varphi$ be the subobject of $M$ classified by $\varphi$. Then for any $a : X \rightarrow M$, define
\begin{equation}\label{DefFor}
X \Vdash \varphi(a) \mathiff a\ \text{factors through}\ S_\varphi.
\end{equation}
This holds iff $\varphi a = \tr_X$, where $\tr_X$ is the arrow $\top\mathbin{\circ} \,{!_X} :X \rightarrow 1 \rightarrow\soc{E}$. 
The idea is that $\varphi$ corresponds to a formula, while $a$ is a generalized element of $M$, thought of as a term $x : X \mid a : M$. In fact, $\varphi$ and $a$ \emph{are} terms in the internal language of $\tps{E}$, reinterpreted into $\tps{E}$ by the forcing relation. The relation $\Vdash$ satisfies certain recursive clauses for all the logical connectives \cite{maclanemoerdijk92, lambekscott88}.
Conversely, starting with an interpretation of the basic symbols of a higher-order type theory in a topos $\tps{E}$ (as maps into $\soc{E}$), then these recursive clauses determine when a formula is true (``at an object $X$''). When $a$ is a closed term, \ie a constant, for which one may assume $X = 1$, then this says that the two arrows 
\[
\xymatrix{
1 \ar[r]^a \ar@/_1.5pc/[rr]_\top & M \ar[r]^\varphi & \soc{E}
}
\]
are equal; \ie the closed sentence $\varphi [a/x]$ is ``true''. In general, the forcing relation thus defines when formulas are true (at $X$), much as  in Kripke semantics, as we now illustrate.

Consider presheaf toposes of the form $\prs{C}$. In this case, the forcing relation $X \Vdash \varphi(a) $ can be restricted to objects $X$ in $\tps{E}$ forming a generating set.\foot{Cf. \cite{lambekscott88}. One says that a set $S$ of objects from $\tps{E}$ is \emph{generating}, iff for any $f \neq g : A \rightrightarrows B$ in $\tps{E}$, there is an arrow $x : X \rightarrow A$, for some $X \in S$, such that $fx \neq gx$.}
For presheaf toposes $\prs{C}$ the representable functors $\y C$ form a generating set, so one may assume that $X = \y C$, for some object $C$ in \cat{C}. Also, by the Yoneda lemma, generalized elements $a:\y C \rightarrow M$ may be replaced by actual elements $a \in M(C)$. To say that $a : \y C \rightarrow M$ factors through a subobject $S \in \sub{\tps{E}}{M}$ is then equivalent to saying that the corresponding element $a \in M(C)$ actually lies in $S(C)$.
As a result, the forcing condition becomes
\[
\y C \Vdash \varphi(a) \mathiff a \in S_\varphi(C),
\]
where, as before, $\varphi$ classifies the subobject $S_\varphi$ of $M$. We shall hereafter write $C \Vdash \dots$ instead of $\y C \Vdash\dots $.

Now consider the standard $\Omega_\ast$-valued model for classical higher-order modal logic in a presheaf topos $\prs{C}$, associated with the canonical geometric morphism $\cat{Sets}^{|\cat{C}|} \rightarrow \prs{C}$.  We define another forcing relation $C \Vdash_* \varphi(a)$ which takes this modal logic into account.

\begin{mydef}\label{DefFor4}
For any presheaf topos $\prs{C}$,  define a \emph{forcing relation} $\Vdash_*$ for arrows $\varphi : M \rightarrow \Omega_\ast$, objects $C$ in \cat{C}, and elements $a \in M(C)$ by:
\begin{equation}\label{DefFor2}
C \Vdash_* \varphi(a) \mathiff C \Vdash\overline{\varphi}(a),
\end{equation}
where $\Vdash$ on the right-hand side is the usual forcing relation \wrt $\cat{Sets}^{|\cat{C}|}$ (as defined in \eqref{DefFor}), and $\overline{(-)}$ indicates transposition along $f^\ast \dashv f_\ast$. 
\end{mydef}

Further analysing the right-hand side of \eqref{DefFor2} gives:
\begin{equation}\label{DefFor3}
C \Vdash \overline{\varphi}(a) \mathiff a \in S_{\overline{\varphi}}(C)
\end{equation}
where $S_{\overline{\varphi}}$ is the subobject of $M^*$ classified by $\overline{\varphi}$ in $\cat{Sets}^{|\cat{C}|}$. 

\begin{prop}
Let $\Vdash_*$ be the forcing relation of Definition \ref{DefFor4}. Then for all $\varphi,\psi : M \rightarrow \Omega_\ast$ and $a \in M(C)$ the following hold:
\begin{align*}
 &C \Vdash_* \top\ &&\text{always}\\
 &C \Vdash_* \bot\ &&\text{never}\\
 &C \Vdash_* \varphi(a) \land \psi(a) &&\mathiff\quad C \Vdash_* \varphi(a) \ \text{and}\ C \Vdash_* \psi(a)
\\
&C \Vdash_* \varphi(a) \lor \psi(a) &&\mathiff\quad C \Vdash_* \varphi(a) \ \text{or}\ C \Vdash_* \psi(a)
\\
 &C \Vdash_* \varphi(a) \Rightarrow \psi(a) &&\mathiff\quad C \Vdash_* \varphi(a) \ \text{implies}\ C \Vdash_* \psi(a)
\\
 &C \Vdash_* \forall x \varphi(x,a) &&\mathiff\quad C \Vdash_* \varphi(b,a)\ \text{for all}\ b \in M(C)
\\
 &C \Vdash_* \exists x \varphi(x,a) &&\mathiff\quad C \Vdash_* \varphi(b,a)\ \text{for some}\ b \in M(C)
\\
 &C \Vdash_* \Box\varphi(a) &&\mathiff\quad D \Vdash_* \varphi(p^\ast a)\ \text{for every}\ p: D \rightarrow C
\\
  &C \Vdash_* t(a) \in u(a) &&\mathiff\quad  (1_C, t_C(a)) \in (u_C(a))_C,\\
 &&&\quad\quad\quad\quad \text{for}\ t : M \rightarrow N \text{and}\ u : M \rightarrow \Omega_\ast^N
\end{align*}
where $\Box = i\tau$, and $\forall x \varphi$ is the arrow $M \xrightarrow{\widehat{\varphi}}\Omega_\ast^M \xrightarrow{\forall_M} \Omega_\ast$, with $\widehat{\varphi}$ the exponential transpose of $M \times M \xrightarrow{\varphi} \Omega_\ast$, and similarly for $\exists x \varphi(x,a)$. 
\end{prop}

\begin{rem}
Although $\Vdash_*$ is a relation between objects $C$ and arrows $\varphi : M \to \Omega_\ast$, it also makes sense to think of the $\varphi$ as formulas, with the clauses above holding \wrt the arrow $\sem{\varphi}$  assigned to the formula $\varphi$ as in section \ref{HOMLModels}. For instance, interpreting a syntactic expression $\exists x \varphi(x,y)$ (by \ref{DefMod}) yields an arrow $\exists_M\widehat{\sem{\varphi}}$.
When $\cat{C}$ is a preorder this is then not merely similar to, but actually \emph{is} the Kripkean satisfaction relation between worlds and formulas, extended to higher-order logic.
\end{rem}

\begin{proof}
We shall just do a few exemplary cases for the purpose of illustration. Consider $C \Vdash_* \varphi(a) \lor \psi(a)$, which by definition \ref{DefFor4} means that $a \in S_{\overline{\varphi\lor\psi}}(C)$. Here, $\Omega_\ast \times \Omega_\ast \xrightarrow{\lor} \Omega_\ast$ is the join map. Recall from proposition \ref{PropRecGeo} that $\lor$ actually is of the form $\lor_\ast$, for the join map $\Omega \times \Omega \xrightarrow{\lor} \Omega$ in $\cat{Sets}^{|\cat{C}|}$. Thus the following commutes, by naturality of the counit $\varepsilon$:
\[
\xymatrix{
M^\ast \times M^\ast \ar[r]^-{\langle \varphi^\ast , \psi^\ast \rangle} \ar[dr]_{\langle \overline{\varphi} , \overline{\psi} \rangle} & (\Omega_\ast)^\ast \times (\Omega_\ast)^\ast \ar[r]^-{(\lor_\ast)^\ast} \ar[d]|{\varepsilon \times \varepsilon} & (\Omega_\ast)^\ast \ar[d]^\varepsilon \\
& \Omega \times \Omega \ar[r]_\lor & \Omega
}
\]
That is to say,
\[
\overline{\varphi\lor\psi} = \overline{\varphi} \lor \overline{\psi},
\]
and so $S_{\overline{\varphi} \lor \overline{\psi}} = S_{\overline{\varphi\lor\psi}}$. Since $\cat{Sets}^{|\cat{C}|}$ is a Boolean topos, by the definition of $S_{\overline{\varphi} \lor \overline{\psi}}$ in $\cat{Sets}^{|\cat{C}|}$ we have: 
\[
a \in S_{\overline{\varphi} \lor \overline{\psi}}(C) \mathiff a \in S_{\overline{\varphi}}(C)\ \text{or}\ a \in S_{\overline{\psi}}(C),
\] 
\ie if and only if $C \Vdash_* \varphi(a) \ \text{or}\ C \Vdash_* \psi(a)$. The argument for the other logical connectives is similar.
\\

For $\forall$, by definition, 
\[
C \Vdash_* \forall x \varphi(x,a) \mathiff a \in S_{\forall_M\widehat{\varphi}}(C),
\] 
with
\[
S_{\forall_M\widehat{\varphi}}(C) = \{a \in M(C) \mid 1_C \in (\forall_M\widehat{\varphi})_C(a)\}
\]
defined as in \eqref{KMods1}.
By the definition of $\forall_M$, and because $|\cat{C}|$ is discrete:
\begin{align*}
1_C \in  (\forall_M\widehat{\varphi})_C(a) & \mathiff 1_C \in \bigcup\{s \in \Omega_\ast(C) \mid \Omega_\ast(f)(s) \leq \widehat{\varphi}_C(a)_D(f,b),\\
&\quad\quad\quad\quad \text{for all}\ f : D \rightarrow C, b \in M(D)\}\\
& \mathiff 1_C \in \bigcup\{s \in \Omega_\ast(C) \mid s \leq \widehat{\varphi}_C(a)_C(1_C,b),\ \text{for all}\  b \in M(C)\}\\
& \mathiff 1_C \in \varphi_C(a,b),\ \text{for all}\  b \in M(C)\\
&\mathiff (a,b) \in S_\varphi,\ \text{for all}\  b \in M(C) \\
&\mathiff C \Vdash_* \varphi(a,b),\ \text{for all}\  b \in M(C).
\end{align*}
The last two equivalences hold by the definition of $S_\varphi$ and $\Vdash_*$. To see the third equivalence, let $\alpha : \y C \rightarrow M$ be the map that corresponds under Yoneda to $a \in M(C)$. Then, by the definition of $\widehat{\varphi}$ (cf. \eqref{DefExTP}):
\[
\widehat{\varphi}_C(a)_C(1_C,b) = \varphi_C(\alpha \times 1_M)_C(1_C,b) = \varphi_C(\alpha_C(1_C) ,b) = \varphi_C(a,b).
\]
Then, if $1_C$ is in the union, it is in one of the $s \in \Omega_\ast(C)$, and thus $1_C \in \varphi_C(a,b)$, for all $b \in M(C)$. On the other hand, if $1_C \in \varphi_C(a,b)$, for all $b \in M(C)$, then $1_C$ is in the union for $s = \{1_C\}$.
\\

The clause for $\in$ follows from its definition:
\begin{align*}
S_{\varepsilon\langle s,t \rangle} & = \{ a \in M(C) \mid 1_C \in \varepsilon\langle s,t\rangle_C(a)\}\\
& =  \{ a \in M(C) \mid 1_C \in  \varepsilon_C(s_C(a),t_C(a)) \}\\
& =  \{ a \in M(C) \mid 1_C \in (s_C(a))_C(1_C,t_C(a))\},	
\end{align*}
using the definition of the evaluation map $\varepsilon : \Omega^A \times A \rightarrow \Omega$.
\\

For $\Box$, as before, $i\tau\varphi$ determines a subfamily of $M$ with components
\[
S_{i\tau\varphi}(C) = \{a \in M(C) \mid 1_C \in (i\tau\varphi)_C(a)\}.
\]
But $(i\tau\varphi)_C(a)$ is a sieve, as it factors through $\Omega(C)$, and so 
\[
S_{i\tau\varphi}(C) = \{a \in M(C) \mid (i\tau\varphi)_C(a) = \top_C\},
\]
for $\top_C$ the maximal sieve on $C$. However, by the defining properties of $\tau$ and $i$, 
\[
(i\tau\varphi)_C(a) = \top_C \mathiff \varphi_C(a) = \top_C.
\]
Therefore, 
\begin{align*}
S_{i\tau\varphi}(C) & = \{a \in M(C) \mid \varphi_C(a) = \top_C\} \\
&= \{a \in M(C) \mid (\chi_{S_\varphi})_C(a) = \top_C\} \\
&= \{a \in M(C) \mid \{p : D \rightarrow C \mid p^\ast a \in S_\varphi(D)\} = \top_C\} \\
& = \{a \in M(C) \mid p^\ast a \in S_{\varphi}(D), \ \text{for all}\ p : D \rightarrow C\}.
\end{align*}
In forcing terms:
\begin{align*}
C \Vdash_* i\tau\varphi(a) & \mathiff a \in S_{i\tau\varphi}(C) \\
& \mathiff p^\ast a \in S_{\varphi}(D), \ \text{for all}\ p : D \rightarrow C\\
& \mathiff D \Vdash_* \varphi(p^\ast a),  \ \text{for all}\ p : D \rightarrow C.
\qedhere
\end{align*}
\end{proof}


%

\begin{exam}\emph{Sheaf Models.}
For a topological space $X$ the (surjective) geometric morphism 
\[
i^\ast\dashv i_\ast : \cat{Sets}/X \longrightarrow \sh{}{X} 
\] 
coming from the continuous inclusion $i : |X| \hookrightarrow X$ gives rise to modal sheaf semantics for classical S4 modal logic as described in \cite{awodeykishida08}. This is most readily seen by viewing sheaves on $X$ as local homeomorphisms over $X$. In this case, the adjunction \eqref{ExtModAd} reads:
\[
\Delta_\pi : \sub{LH/X}{E} \leftrightarrows \sub{\cat{Sets}/X}{i^\ast E} : \Gamma_\pi
\] 
where $E \rightarrow X$ is a local homeomorphism.  A subobject of $i^\ast E$ in $\cat{Sets}/X$  is simply a commutative triangle of functions in $\cat{Sets}$ 
\[
\xymatrix{
A \ar@{^(->}[rr] \ar[dr] && E \ar[dl]\\
&X&
}
\]
which is entirely determined by a subset $A \subseteq E$. One obtains the largest subsheaf of $E$ contained in $A$ just by applying the interior operator of $E$ to $A \subseteq E$:
\[
\xymatrix{
\text{int} A \ar@{^(->}[rr] \ar[dr] && E \ar[dl]\\
&X&
}
\]
The horizontal inclusion is then continuous \wrt the subspace topology on $\text{int} A$.
The composite is then a local homeomorphism, because the restriction of any local homeomorphism to an open subset of the total space ($E$) is one. 

This is therefore just the familiar topological semantics for \emph{propositional} modal logic, given by the adjunction
 \[
i : \sub{\sh{}{X}}{E} \cong \mathcal{O}(E) \leftrightarrows \pow{E} \cong \sub{\cat{Sets}/X}{i^\ast E} : \text{int}
\]

In this case the algebraic formulation via maps into the subobject classifier is perhaps less intuitive. The subobject classifier $\omega : \Omega \rightarrow X$ in $\sh{}{X}$ has the fibers:\foot{See  \eg \cite{maclanemoerdijk92}.}
\[
\inv{\omega}(x) = \varinjlim_{x \in U}\downarrow\!{U}
\] 
where $\downarrow\!{U}$ is the set of all open subsets of $U \in \ops{X}$.  On the other hand, viewing sheaves as a special kind of presheaves, the formulation is now more familiar.  The subobject classifier takes the form $\Omega_X(U) = \downarrow\!{U}$ (for $V \subseteq U$ this acts by $V \cap -$, \ie the inverse image along the inclusion). Thus $\Omega_X(U) = \ops{U}$ for the subspace topology on $U$. In turn, $\Omega_\ast(U) = \pow{U}$ with the evident restriction along inclusions. Thus propositions are modelled by natural transformations
$M \rightarrow \mathcal{P}$
to the contravariant powerset-functor, while the map $\tau_U : \pow{U} \rightarrow \ops{U}$, for any $U \subseteq X$, picks the largest open subset contained in a given subset of $U$, i.e.\ the interior.


With this description, sheaf semantics may be seen as the generalization of the familiar topological semantics for \emph{propositional} modal logic to \emph{quantified} languages.
%
The previous case of presheaves on a preorder $\cat{K}$ is actually a special case of this one by taking the Alexandroff topology on $\cat{K}$.
\end{exam}


\section{Geometric models from algebraic ones}\label{ReprAlgMod}

The foregoing shows that every geometric model gives rise to a logically equivalent algebraic model in the sense of section \ref{HOMLModels}. The following observation, obtained through general topos-theoretic considerations, states the converse.

\begin{fact}
For any complete Heyting algebra $H$ in a topos $\tps{E}$, the canonical structure 
\[
\tau : H \leftrightarrows \soc{E} : i
\]
($i \dashv \tau$) arises from a topos $ \tps{H}$ and geometric morphism $g : \tps{H} \rightarrow \tps{E}$, via $H = \gsoc{g}{H}$.
\end{fact}

\begin{proof} (sketch)
The topos $\tps{H}$ may be defined as the category $\sh{\tps{E}}{H}$ of internal sheaves on $H$. A description of $\sh{\tps{E}}{H}$ can be given in terms of locales in $\tps{E}$ (see \cite{johnstone02} C1.3).  A local homeomorphism over the locale $H$ is an open locale map $E \rightarrow H$ with open diagonal $E \rightarrow E \times_H E$, where the codomain is the product of locale morphisms over $H$ (in $\tps{E}$). This is an internalization of the notion of local homeomorphism over the ``space'' $H$, in view of the fact that a continuous map $\pi : Y \rightarrow X$ of topological spaces is a local homeomorphism just in case both $\pi$ and its diagonal (over $X$) are open maps. Alternately, using the internal language of $\tps{E}$, the category $\sh{\tps{E}}{H}$ may be described as consisting of internal presheaves on the site $H$ (with the sup-topology) that satisfy the usual sheaf property in the internal language. See \cite{johnstone02}, C1.3 for details.

Next, recall that for any two frames $X,Y$ in $\tps{E}$, there is an equivalence of categories
\begin{equation}\label{GeomMod1}
\cat{Fr}_\tps{E}(Y,X) \simeq \cat{Top}(\sh{\tps{E}}{X} , \sh{\tps{E}}{Y})
\end{equation}
between frame homomorphisms $Y \rightarrow X$ in $\tps{E}$ and geometric morphisms $\sh{\tps{E}}{X} \rightarrow \sh{\tps{E}}{Y}$ \cite{johnstone02, maclanemoerdijk92}. Then $g : \sh{\tps{E}}{H} \rightarrow \tps{E}$ arises under this equivalence from the frame map $i$, noting that 
\[
\tps{E} \simeq \sh{\tps{E}}{\soc{E}}.
\]
Externally, the idea of \eqref{GeomMod1} is that the inverse image part $g^\ast$ of a geometric morphism $g : \sh{}{X} \rightarrow \sh{}{Y}$ restricts to a frame homomorphism
\[
g^\ast : \sub{\sh{}{Y}}{1} \rightarrow \sub{\sh{}{X}}{1},
\] 
where 1 is the terminal object, respectively. Observing that for any sheaf topos $\sh{}{X}$, we have $\sub{\sh{}{X}}{1} \cong \ops{X}$ gives the required frame map. On the other hand, it is also well-known that a frame map $Y \rightarrow X$ induces a geometric morphism of the required form for the sup-topology on $X$ and $Y$, respectively. These constructions are inverse and relativize to an arbitrary topos $\tps{E}$ instead of the usual category of \cat{Sets} \cite{johnstone02,JoyalTierney84}.
Moreover, the geometric morphism $g$ is surjective if $i$ is monic. 

Lastly,
\[
H \cong g_\ast \Omega_{\sh{\tps{E}}{H}},
\]
 because $\sh{\tps{E}}{H}$ coincides with the hyperconnected-localic factorization of $g$ itself, which is determined (up to equivalence of categories) \cite{JoyalTierney84} as the sheaf topos
\[
\sh{\tps{E}}{{g_\ast\Omega_{\sh{\tps{E}}{H}}}},
\]
whence it follows that 
\[
H \cong \sub{\sh{\tps{E}}{H}}{1} \cong g_\ast\Omega_{{Sh}_{\tps{E}}(H)}.
\qedhere
\]
\end{proof}

This last observation applies in particular in case $H = \gsoc{f}{F}$ is already of the required form. Then $Sh_\tps{E}(\Omega_\ast)$ occurs in the hyperconnected-localic factorization of~$f$:
\[
\xymatrix{
\tps{F} \ar[rr] \ar[dr]_f &&Sh_\tps{E}(\Omega_\ast) \ar[dl]^g\\
&\tps{E}&\\
}
\]
and
\[
\gsoc{f}{F} \cong g_\ast\Omega_{\sh{\tps{E}}{\gsoc{f}{F}}}.
\]
Externally, we have:
\begin{align*}
\subt{F}{f^\ast A} & \cong \Hom{\tps{F}}{f^\ast A}{\Omega_{\tps{F}}}\\
&\cong \Hom{\tps{E}}{A}{\gsoc{f}{F})}\\
& \cong \Hom{\tps{E}}{A}{g_\ast\Omega_{\sh{\tps{E}}{\gsoc{f}{F}}}}\\
& \cong \Hom{\sh{\tps{E}}{\gsoc{f}{F}}}{g^\ast A}{\Omega_{\sh{\tps{E}}{\gsoc{f}{F}}}}\\
& \cong \sub{\sh{\tps{E}}{\gsoc{f}{F}}}{g^\ast A}
\end{align*}
for all $A$ in $\tps{E}$. This allows us to restrict attention to \emph{localic} surjective geometric morphisms. For instance, the geometric morphism 
\[
i^\ast \dashv i_\ast : \cat{Sets}^{|\cat{D}|} \rightarrow \srp{D}
\]
considered in the previous section is localic.

\bibliographystyle{acm}
\bibliography{ModalThesis}

\end{document}